\documentclass{gOPT2e}

\usepackage{subfigure}% Support for small, `sub' figures and tables
\usepackage{enumerate}% added by authors
\usepackage{float}% added by authors to use "[H]"

\theoremstyle{plain}
\newtheorem{theorem}{Theorem}[section]

\newtheorem{corollary}[theorem]{Corollary}
\newtheorem{lemma}[theorem]{Lemma}
\newtheorem{proposition}[theorem]{Proposition}

\theoremstyle{remark}
\newtheorem{remark}[theorem]{Remark} % counting like theorems, changed by authors

\theoremstyle{definition}
\newtheorem{definition}[theorem]{Definition}
\newtheorem{example}[theorem]{Example} % added by authors

\DeclareMathOperator*{\interior}{int}
\DeclareMathOperator*{\fix}{Fix}
\DeclareMathOperator*{\id}{Id}
\DeclareMathOperator*{\Argmin}{Argmin}
\DeclareMathOperator*{\argmax}{argmax}
\begin{document}

%\jvol{00} \jnum{00} \jyear{2014} \jmonth{February}

%\articletype{GUIDE}

\title{{\itshape Outer Approximation Methods for Solving Variational Inequalities in Hilbert Space}}

\author{
Aviv Gibali$^{\rm a}$, Simeon Reich$^{\rm b}$ and Rafa\l\ Zalas$^{\rm b}$$^{\ast}$\thanks{$^\ast$Corresponding author: Rafa\l\ Zalas, Email: rzalas@tx.technion.ac.il}\\\vspace{6pt}  $^{a}${\em{Department of Mathematics, ORT Braude College, 2161002 Karmiel, Israel;}}
$^{b}${\em{Department of Mathematics, The Technion - Israel Institute of Technology, 3200003 Haifa, Israel.}}\\
\received{v4.1 released February 2014} }

\maketitle

\begin{abstract}
In this paper we study variational inequalities in a real Hilbert space, which are governed by a strongly monotone and Lipschitz continuous operator $F$ over a closed and convex set $C$. We assume that the set $C$ can be outerly approximated by the fixed point sets of a sequence of certain quasi-nonexpansive operators called cutters. We propose an iterative method the main idea of which is to project at each step onto a particular half-space constructed by using the input data. Our approach is based on a method presented by Fukushima in 1986, which has recently been extended by several authors. In the present paper we establish strong convergence in Hilbert space. We emphasize that to the best of our knowledge, Fukushima's method has so far been considered only in the Euclidean setting with different conditions on $F$. We provide several examples for the case where $C$ is the common fixed point set of a finite number of cutters with numerical illustrations of our theoretical results.


\begin{classcode}
47H09; 47H10; 47J20; 47J25; 65K15.
\end{classcode}

\end{abstract}

\section{Introduction}

Let $(\mathcal H, \langle\cdot,\cdot\rangle)$ be a real Hilbert space with induced norm $\|\cdot\|$.
The \textit{variational inequality} VI($F$, $C$) governed by a monotone operator $ F%
\colon\mathcal H\rightarrow\mathcal H$ over a nonempty, closed and
convex set $C\subseteq\mathcal H$ is formulated as the following problem:
find a point $x^*\in C$ for which the inequality%
\begin{equation}
\langle F x^*, z-x^*\rangle\geq 0
\label{eq:vip}
\end{equation}
holds true for all $z\in C$. In the last decades VIs have been extensively studied by many authors; see, for example, Facchinei's and Pang's two-volume book \cite{FacchineiPang2003}, the review papers by Xiu and Zhang \cite{XiuZhang2003}, and Noor \cite{Noor2004}, as well as a recent one by Chugh and Rani \cite{ChughRani2014}.

It is not difficult to see, compare with \cite[Theorem 1.3.8]{Cegielski2012}, that $x^*$ solves VI($F$,$C$) if and only if it satisfies the fixed point equation $x^*=P_C(x^*-\lambda Fx^*)$ for some $\lambda>0$.  Moreover, if $F$ is $L$-Lipschitz continuous and $\alpha$-strongly monotone, then the operator $P_C(\id -\lambda F)$ becomes a strict contraction for any $\lambda\in (0, \frac{2\alpha}{L^2})$; see, for example, either \cite[Theorem 46.C]{Zeidler1985} or \cite[Theorem 5]{CegielskiZalas2013a}. Therefore, by Banach's fixed point theorem, the VI($F$, $C$) has a unique solution. Moreover, in order to approximate this solution, one could try to apply a fixed point iteration of the form
\begin{equation}\label{eq:def:GP}
x^0\in\mathcal H;\qquad x^{k+1}:=P_C(x^k-\lambda F x^k), \text{ for } k=0,1,2,\ldots,
\end{equation}
which by the same argument is known to converge strongly to $x^*$. This method appeared in the literature as the \textit{gradient projection method} in the context of minimization and was introduced by Goldstein \cite{Goldstein1964}, and Levitin and Polyak \cite{LevitinPolyak1966}.

The gradient projection method can be particularly useful when estimates of the constants $L$, $\alpha$, and thereby $\lambda$, are known in advance, and when the set $C$ is simple enough to project onto. However, in general, this does not have to be the case and therefore the efficiency of the method can be essentially affected. To overcome the first obstacle, one can replace $\lambda$ with an unknown estimate by a null, non-summable sequence $\{\lambda_k\}_{k=0}^\infty\subseteq[0,\infty)$. To overcome the other difficulty, one can replace the metric projection onto $C$ by a sequence of metric projections onto certain half-spaces $H_k$ containing $C$, which should be simpler to calculate. This leads to the \textit{outer approximation method}
\begin{equation}\label{eq:def:xk:intro}
x^0\in\mathcal H;\qquad x^{k+1}:=R_k(x^k-\lambda_k F x^k), \text{ for } k=0,1,2,\ldots,
\end{equation}
where $R_k:=\id+\alpha_k(P_{H_k}-\id)$ and $\alpha_k\in[\varepsilon,2-\varepsilon]$ is the user-chosen relaxation parameter. The characteristics and, in particular, the computational cost of such methods depend to a large extent on the construction of the half-space $H_k$. A common feature of these methods is that the boundary of $H_k$ should separate $x^k$ from $C$ whenever $x^k\notin C$ and $H_k=\mathcal H$ otherwise. Such an approach has been successfully applied several times and can be found in the literature.

For instance, Fukushima \cite{Fukushima1986} defined $H_k:=\{z\in\mathbb R^n\mid f(x^k)+\langle g_f(x^k), z-x^k\rangle \leq 0\}$, assuming that $C$ is the sublevel set at level $0$ of a convex function $f\colon\mathbb R^n\rightarrow \mathbb R$, that is, $C=\{x\in\mathbb R^n\mid f(x)\leq 0\}$, and that for each $k=0,1,2,\ldots,$ the vector $g_f(x^k)$ is a subgradient of $f$ at $x^k$, that is, $g_f(x^k)\in\partial f(x^k)$.

Censor and Gibali \cite{CensorGibali2008} proposed a similar approach, but with a more flexible choice of the half-space $H_k$. In this case the boundary of $H_k$ should separate a ball $B(x^k,\delta d(x^k,C))$ from $C$, where $0<\delta\leq 1$. It turns out that the boundary of Fukushima's half-space $H_k$, defined via a subgradient of $f$, separates $B(x^k, \delta f(x^k))$ from $C$ for some $\delta \in (0,1]$; see \cite[Lemma 2.8]{CensorSegal2009}.

Cegielski \textit{et al.} \cite{CegielskiGibaliReichZalas2013} constructed the half-space $H_k$ by exploiting the structure of the set $C$, which in their case was represented as a fixed point set $C=\fix  T:=\{z\in\mathcal \mathbb R^n\mid z=Tz\}$ of an operator $T\colon\mathbb R^n\rightarrow \mathbb R^n$. This operator was assumed to be a weakly regular cutter, that is, $T-\id$ is demi-closed at 0 (see Definition 2.9) and $\langle x-Tx, z-Tx \rangle \leq 0$ for all $x\in \mathcal H$ and $z\in \fix T$. Here $H_k:=\{z\in\mathcal H\mid \langle x^k-Tx^k,z-Tx^k\rangle \leq 0\}$, where the separation of $x^k$ from the boundary of $H_k$ was assured by restricting the choice of $T$ to cutters. In particular, by setting $T$ to be a subgradient projection $P_f$, which is also a cutter (see Example \ref{ex:def:SubPro}), we recover Fukushima's half-space. In addition, one can easily show that $P_f$ is weakly regular whenever the dimension of $\mathcal H$ is finite (see Example \ref{ex:SubProjandDCPrinciple}). Moreover, this concept is more general than the one from \cite{CensorGibali2008}; a detailed explanation can be found in \cite[Example 2.22]{CegielskiGibaliReichZalas2013}.

In this direction, Gibali \textit{et al.} \cite{GibaliReichZalas2015} have recently considered VIs with a subset $C$ outerly approximated by an infinite family of cutters $T_k\colon\mathbb R^n\rightarrow\mathbb R^n$ in the sense that \begin{equation}
C\subseteq\bigcap_{k=0}^\infty\fix  T_k.
\label{eq:outerAppxCFixTk}
\end{equation}
In the definition of $H_k$ the constant operator $T$ was replaced by a sequence of operators $\{T_k\}_{k=0}^\infty$, that is, $H_k:=\{z\in\mathcal H\mid \langle x^k-T_kx^k,z-T_kx^k\rangle \leq 0\}$. A general regularity condition \cite[Condition 3.6]{GibaliReichZalas2015} which is somewhat related to weak regularity and therefore to the demi-closedness principle has been imposed on the family of $T_k$'s. In particular, for $C$ defined as the solution set of the common fixed point problem for a finite family of weakly regular cutters $U_i\colon\mathbb R^n\rightarrow\mathbb R^n$, $i\in I=\{1,\ldots,m\}$, the operators $T_k$ were defined by using either a cyclic ($T_k=U_{[k]}$, $[k]= (k \text{ mod } m) +1$), simultaneous ($T_k=\sum_{i\in I_k}\omega_i^k U_i$) or a composition ($T_k=(\id+\prod_{i\in I_k}U_i$)/2) algorithmic operators.

Another slightly different, but still a strongly related approach has been considered by Cegielski and Zalas \cite{CegielskiZalas2013a}, where the following \textit{hybrid steepest descent method} (HSD)
\begin{equation}\label{eq:def:HSD}
z^0\in\mathcal H;\qquad z^{k+1}:=R_kz^k-\lambda_k F R_kz^k, \text{ for } k=0,1,2,\ldots,
\end{equation}
has been investigated for VI($F$,$C$) with a Lipschitz continuous and strongly monotone $F$ defined over a closed and convex $C$ in an infinite dimensional Hilbert space. Here $C$ was assumed to be outerly approximated, like in (\ref{eq:outerAppxCFixTk}), by a sequence of strongly quasi-nonexpansive operators $R_k$ and, in particular, by cutters. The HSD method was originally proposed by Deutsch and Yamada \cite{DeutschYamada1998} in a simpler setting, although its origin goes back to Halpern's paper \cite{Halpern1967} from 1967, where $F=\id-a$. Various instances of the HSD method have been studied in the meantime. For example, Lions \cite{Lions1977}, Wittmann \cite{Wittmann1992}, Bauschke \cite{Bauschke1996} and Slavakis \textit{et al.} \cite{SlavakisYamadaSakaniwa2003} have considered the HSD method with $F=\id-a$. On the other hand, the HSD method for more general $F$ has been investigated by Yamada \cite{Yamada2001}, Xu and Kim \cite{XuKim2003}, Yamada and Ogura \cite{YamadaOgura2004}, Hirstoaga \cite{Hirstoaga2006}, Zeng \textit{et al.} \cite{ZengWongYao2007}, Yamada and Takahashi \cite{TakahashiYamada2008}, Aoyama and Kimura \cite{AoyamaKimura2011}, Zhang and He \cite{ZhangHe2013}, Cegielski and Zalas \cite{CegielskiZalas2014}, Aoyama and Kohsaka \cite{AoyamaKohsaka2014}, Zalas \cite{Zalas2014}, Cegielski \cite{Cegielski2014, Cegielski2015}, and Cegielski and Al-Musallam \cite{CegielskiMusallam2016}.

It turns out that iteration (\ref{eq:def:xk:intro}) can be viewed as the HSD method (\ref{eq:def:HSD}) with $R_k:=\id+\alpha_k(P_{H_k}-\id)$ (see Section \ref{sec:OAM:ConvAnal}). This suggests that similar sufficient conditions for the strong convergence of (\ref{eq:def:HSD}) should apply to (\ref{eq:def:xk:intro}) in the infinite dimensional setting. We emphasize here that convergence results for the iterative method (\ref{eq:def:xk:intro}), to the best of our knowledge, have so far been established in Euclidean space only, also by imposing global conditions on $F$ different than Lipschitz continuity; compare with \cite[Assumption (c)]{Fukushima1986}, \cite[Condition 5]{CensorGibali2008}, \cite[Condition 3.3]{CegielskiGibaliReichZalas2013} and \cite[Condition 3.4]{GibaliReichZalas2015}.

In the present paper we assume that $F$ is Lipschitz continuous and strongly monotone, and that $C$ is outerly approximated (compare with (\ref{eq:outerAppxCFixTk})) by an infinite sequence of cutters $T_k\colon\mathcal H\rightarrow\mathcal H$. The main contribution of our paper is to provide sufficient conditions for the strong convergence of method (\ref{eq:def:xk:intro}) in a general real Hilbert space $\mathcal H$ (Theorem \ref{th:main}). In particular, for $C$ defined as the solution set of the common fixed point problem with respect to a finite family of operators $U_i\colon\mathcal H\rightarrow\mathcal H$, following \cite{GibaliReichZalas2015}, we allow the $T_k$'s to be defined either by cyclic, simultaneous or composition algorithmic operators. Moreover, we permit not only cyclic but also maximum proximity algorithmic operators which are more general than the remotes-set and most-violated constraint case. These results are summarized in Theorems \ref{th:main2} and \ref{th:main3}.

This paper is organized as follows. Section \ref{sec:preliminaries} is a preliminary section, where we recall several definitions, examples and theorems to be used in the rest of the paper. In Subsection \ref{sec:HSD} we discuss convergence properties of the HSD method (\ref{eq:def:HSD}). Section \ref{sec:OAM} contains general convergence results regarding the outer approximation method (\ref{eq:def:xk:intro}), while Section \ref{sec:VIoverCFPP} provides applications of this result to VIs defined over the solution set of a common fixed point problem. In the last section we provide some numerical results which illustrate the validity of our theoretical analysis.

\section{Preliminaries} \label{sec:preliminaries}

Let $C\subseteq\mathcal H$ and $x\in\mathcal H$ be given. If there is a point $y\in C$ such that $\| y-x\|\leq\| z-x\|$ for all $z\in C$, then $y$ is called a \textit{metric projection} of $x$ onto $C$ and is denoted by $P_{C} x $. If $C$ is nonempty, closed and convex, then for any $x\in\mathcal H$, the metric projection of $x$ onto $C$ exists and is uniquely defined; see, for example, \cite[Theorem 1.2.3]{Cegielski2012}. In this case the function $d(\cdot,C):\mathcal H \rightarrow[0,\infty)$ measuring the distance between an arbitrary given $x \in\mathcal H$ and $C$ satisfies $d(x,C)=\| P_{C} x -x\|$.

For a given $U\colon\mathcal H\rightarrow\mathcal H$ and $\alpha
\in (0,\infty)$, the operator $U_{\alpha}:=\id +\alpha(U-\id )$ is called an $\alpha$-\textit{relaxation of} $U$, where by $\id $ we denote the identity operator. We call $\alpha$ a \textit{relaxation parameter}. It is easy to see that for every $\alpha\neq 0$, $\fix U = \fix U_\alpha$, where we recall that $\fix U:=\{z\in\mathcal H\mid Uz=z\}$ is the \textit{fixed point set} of $U$. Usually, in connection with iterative methods, as in (\ref{eq:def:xk:intro}), the relaxation parameter $\alpha$ is assumed to belong to the interval $[\varepsilon,2-\varepsilon]$.

\subsection{Quasi-nonexpansive and nonexpansive operators}\label{sec:QNE}

\begin{definition}\label{def:QNE}
Let $U:\mathcal H\rightarrow\mathcal H$ be an
operator with a fixed point, that is, $\fix U
\neq \emptyset$. We say that $U$ is
\begin{itemize}
\item \textit{quasi-nonexpansive} (QNE) if for all $x\in\mathcal H$ and all $z\in\fix U$,
\begin{equation}
\| U x -z\|\leq\| x-z\|;\label{eq:qne}%
\end{equation}

\item  $\rho$\textit{-strongly quasi-nonexpansive} ($\rho$-SQNE), where $\rho\geq 0$, if for all $x\in\mathcal H$ and all $z\in\fix U$,
\begin{equation}
\| Ux-z\|^2\leq\| x-z\|^2-\rho\| U
x -x\|^2;\label{eq:sqne}
\end{equation}

\item a \textit{cutter} if for all $x\in\mathcal H$ and all $z\in\fix U$,
\begin{equation}
\langle z-U x ,x-U x \rangle\leq 0\label{eq:cutter}.
\end{equation}
\end{itemize}
\end{definition}
See Figure \ref{fig:cutter} for the geometric interpretation of a cutter.

\begin{figure}[htb]
\centering
\includegraphics[scale=0.25]{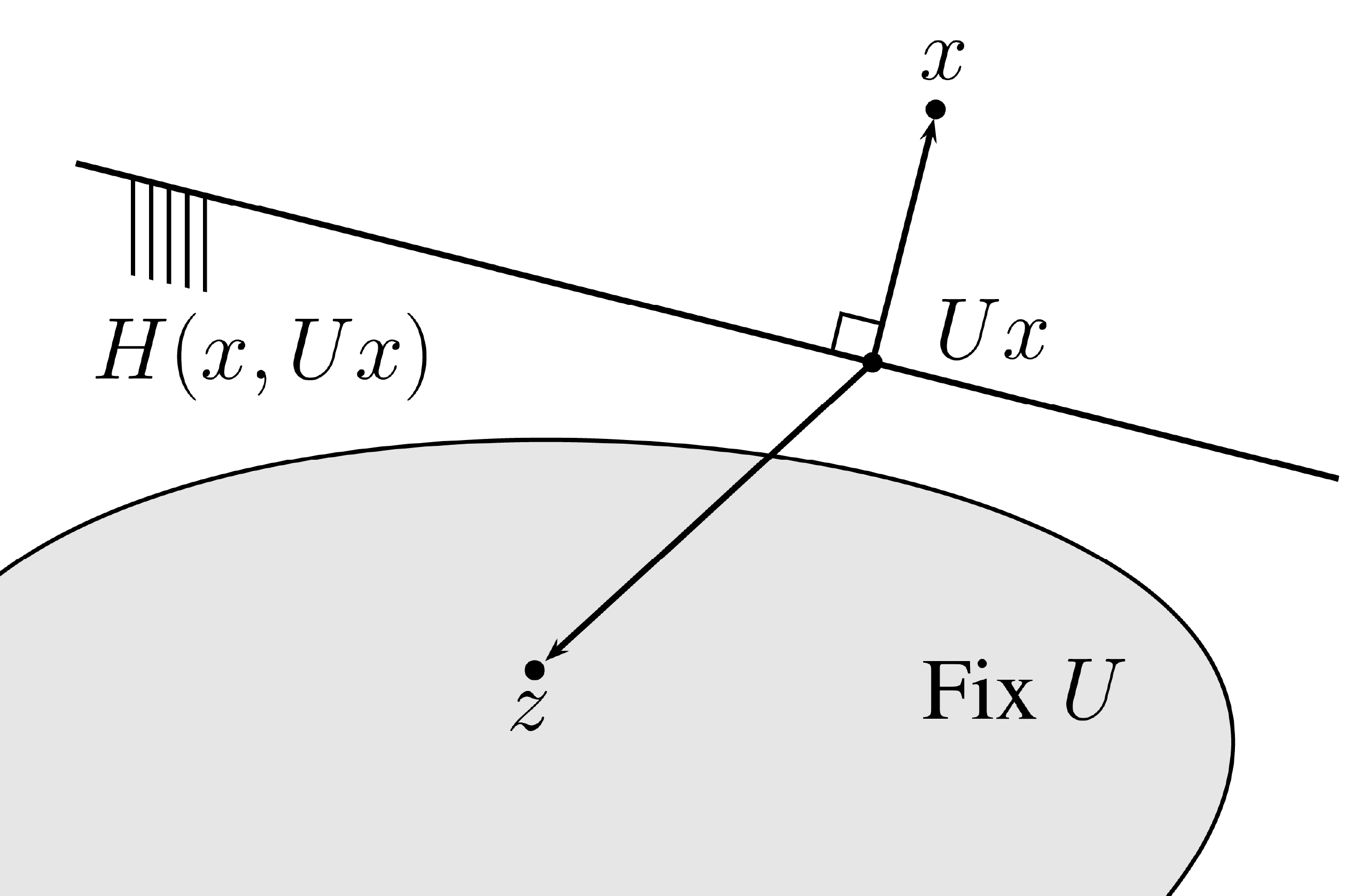}
\caption{Geometric interpretation of a cutter $U$. Note that for every $x\in\mathcal H$ we have $\fix U\subseteq H(x,Ux):=\{z\in\mathcal H\mid \langle z-Ux, x-Ux\rangle\leq 0\}$.}
\label{fig:cutter}
\end{figure}

\begin{definition}\label{def:NE}
Let $U\colon\mathcal H\rightarrow\mathcal H$. We say that $U$ is
\begin{itemize}
\item \textit{nonexpansive} (NE) if for all $x,y\in\mathcal H$,
\begin{equation}\label{eq:def:NE}
\|Ux-Uy\|\leq\|x-y\|;
\end{equation}

\item $\rho$-\textit{firmly nonexpansive} ($\rho$-FNE) \cite[Definition 2.1]{Crombez2006}, where $\rho \geq 0$, if for all $x,y\in\mathcal H$,
\begin{equation}\label{eq:def:FNEwithRho}
\|Ux-Uy\|^2\leq\|x-y\|^2-\rho\|(Ux-x)-(Uy-y)\|^2;
\end{equation}

\item \textit{firmly nonexpansive} (FNE) if for all $x,y\in\mathcal H$,
\begin{equation}\label{eq:def:FNE}
\langle Ux-Uy, x-y\rangle \geq\|Ux-Uy\|^2.
\end{equation}

\end{itemize}
\end{definition}

For a historical overview of the above-mentioned operators we refer the reader to \cite{Cegielski2012}. We have the following theorems.

\begin{theorem}
\label{th:cuttersAndQNE}
Let $U\colon\mathcal H\rightarrow\mathcal H$ be an operator with a fixed point and let $\alpha\in(0,2]$. Then $U$ is a cutter if and only if its relaxation $\id +\alpha(U-\id )$ is $(2-\alpha)/\alpha$-strongly quasi-nonexpansive.
\end{theorem}

\begin{proof}
See, for example, either \cite[Proposition 2.3(ii)]{Combettes2001} or \cite[Theorem 2.1.39]{Cegielski2012}.
\end{proof}

\begin{corollary}
\label{th:cuttersAndQNE:corollary} Let $U\colon\mathcal H\rightarrow\mathcal H$
be an operator with a fixed point and let $\rho\geq 0$. Then $U$ is $\rho$-SQNE if and only if its relaxation $\id +\frac{1+\rho}{2}(U-\id )$ is a cutter.
\end{corollary}

\begin{proof}
See, for example, \cite[Corollary 2.1.43]{Cegielski2012}.
\end{proof}

\begin{theorem}
\label{th:FNEandNE} Let $U\colon\mathcal H\rightarrow\mathcal H$
be an operator and let $\alpha\in(0,2]$. Then $U$ is firmly nonexpansive (in the sense of (\ref{eq:def:FNE})) if and only if its relaxation $\id + \alpha(U-\id )$ is $(2-\alpha)/\alpha$-firmly nonexpansive.
\end{theorem}

\begin{proof}
For $\alpha \in (0,2)$, see, for example, \cite[Corollary 2.2.15]{Cegielski2012} and for  $\alpha=2$, see \cite[Theorem 2.2.10]{Cegielski2012}.
\end{proof}

Let $U\colon\mathcal H\rightarrow\mathcal H$ be an operator with $\fix U\neq\emptyset$.
One can easily see that if $U$ is NE, then it is QNE. Similarly, $U$ is $\rho$-SQNE whenever it is $\rho$-FNE. In addition, by Theorem \ref{th:cuttersAndQNE}, a cutter $U$ is $1$-SQNE and by Theorem \ref{th:FNEandNE}, an FNE operator $U$ is $1$-FNE. Hence an FNE $U$ is a cutter. Furthermore, $U$ is QNE if and only if $(\id +U)/2$ is a cutter. In the same manner $U$ is NE if and only if $(\id +U)/2$ is FNE.

The set of fixed points of a cutter $U$ is closed and convex. Moreover,
$ \fix  U =\bigcap_{x\in\mathcal H}H(x,U x ),$ where $H(x,U x ):=\{z\in\mathcal H\mid\langle z-U x ,x-U x \rangle\leq 0\}$; see \cite[Proposition
2.6(ii)]{BauschkeCombettes2001}. Therefore, by the relation $\fix U=\fix U_\alpha$, Theorems \ref{th:cuttersAndQNE} and \ref{th:FNEandNE}, the set $\fix U$ is closed and convex whenever $U$ is either QNE, SQNE, NE or FNE.

\begin{example}\label{ex:MetProj}
The metric projection onto a nonempty, closed and convex set $C$ is FNE \cite[Theorem 2.2.21]{Cegielski2012}. Since $\fix P_C=C\neq\emptyset$, it is also a cutter. Therefore, for any $\alpha\in(0,2]$, the relaxation
$\id +\alpha(P_C-\id )$ is $\rho$-FNE and $\rho$-SQNE, where $\rho:=(2-\alpha)/\alpha$. Consequently, this relaxation is also NE and QNE.
\end{example}

\begin{example} \label{ex:def:SubPro}
Let $f:\mathcal H\rightarrow \mathbb{R}$ be a convex continuous function with a nonempty sublevel set $S(f, 0):=\{x\mid f(x)\leq 0\}$. Denote by $\partial f(x)$ its subdifferential, that is, $\partial f(x):=\{g\in\mathcal H
\mid f(y)-f(x)\geq\langle g,y-x\rangle\text{ for all } y\in\mathcal H\}$. By the continuity of $f$, the set $\partial f(x)\neq\emptyset$ for all $x\in \mathcal H$ (see \cite[Proposition 16.3 and Proposition 16.14]{BauschkeCombettes2011}). For each $x\in \mathcal H$, let $g_f(x)\in\partial f(x)$ be a given subgradient. The so-called \textit{subgradient projection} relative to $f$ is the operator $P_f:\mathcal H\rightarrow\mathcal H$ defined by
\begin{equation} \label{eq:ex:def:SubPro}
P_f x:=
\begin{cases}
x-\frac{{f(x)}}{\| g_f(x)\|^2}g_f(x) & \text{ if $g_f(x)\neq 0$,}\\
x & \text{ otherwise.}%
\end{cases}
\end{equation}
It is not difficult to see that $\fix P_f=S(f,0)$ (see \cite[Lemma 4.2.5]{Cegielski2012}) and that $P_f$ is a cutter (see \cite[Corollary 4.2.6]{Cegielski2012}). Moreover, one may replace the condition \textquotedblleft$g_f(x)\neq 0$" in the definition of $P_f$ by the condition \textquotedblleft$f(x)>0$", which leads to an equivalent definition of the subgradient projection. Similarly, as in the previous example, the relaxation
$\id +\alpha(P_f-\id )$ is $\rho$-SQNE, where $\rho:=(2-\alpha)/\alpha$.
\end{example}

The next theorem provides a relations between given SQNE operators
$U_{1},\ldots,U_{m}$ and their convex combinations or compositions.

\begin{theorem}
\label{th:SQNE} Let $U_{i}:\mathcal H\rightarrow\mathcal H$ be
$\rho_{i}$-strongly quasi-nonexpansive, $i\in I:=\{1,\ldots,m\}$, with $\bigcap_{i\in I}\fix  U_{i} \neq\emptyset$ and let $\rho:=\min_{i\in I}\rho_i>0$. Then:
\begin{enumerate}[(i)]
\item the convex combination $U:=\sum_{i\in I}\omega_{i}U_{i}$, where
$\omega_{i}>0,\ i\in I$, and $\sum_{i\in I}\omega_{i}=1$, is $\rho$-strongly quasi-nonexpansive;
\item the composition $U:=U_{m}\ldots U_{1}$ is $\frac{\rho}{m}$-strongly quasi-nonexpansive.
\end{enumerate}
Moreover, in both cases,
\begin{equation}
\fix  U =\bigcap_{i\in I}\fix U_{i}.
\end{equation}
\end{theorem}

\begin{proof}
See, for example, \cite[Theorems 2.1.48 and 2.1.50]{Cegielski2012}.
\end{proof}

\subsection{Regular operators} \label{sec:RegOperators}
\begin{definition}\label{def:WRBR}
We say that a quasi-nonexpansive operator $U:\mathcal H\rightarrow \mathcal
H$ is
\begin{enumerate}[(i)]
  \item \textit{weakly regular} (WR) if $U-\id$ is demi-closed at 0, that is, if for any sequence $\{x^k\}_{k=0}^\infty \subseteq\mathcal H$ and $x\in\mathcal H$, we have
    \begin{equation}\label{eq:def:WR}
    \left.
    \begin{array}
    [c]{l}%
    x^k\rightharpoonup x\\
    U x^k -x^k\rightarrow 0
    \end{array}
    \right\} \Longrightarrow x\in\fix U.
    \end{equation}
  \item \textit{boundedly regular} (BR) if for any bounded sequence $\{x^{k}\}_{k=0}^\infty \subseteq \mathcal H$, we have
  \begin{equation} \label{eq:def:BR}
    \lim_{k\rightarrow\infty}\| Ux^k-x^k\| =0\quad\Longrightarrow\quad \lim_{k\rightarrow\infty}d(x^k,\fix U)=0.
    \end{equation}
\end{enumerate}

\end{definition}
Weakly regular operators go back to papers by Browder and Petryshyn \cite{BrowderPetryshyn1966} and by Opial \cite{Opial1967}. A prototypical version of condition \eqref{eq:def:BR} can be found in \cite[Theorem 1.2]{PetryshynWilliamson1973} by Petryshyn and Williamson. The term ``boundedly regular'' comes from \cite{BauschkeNollPhan2015} by Bauschke, Noll and Phan while the term ``weakly regular'' can be found in \cite{KolobovReichZalas2016} by Kolobov, Reich and Zalas. Boundedly regular operators have been studied under the name \textit{approximately shrinking} in \cite{CegielskiZalas2013a, CegielskiZalas2014, Zalas2014, Cegielski2015, CegielskiMusallam2016, ReichZalas2015}, whereas weakly regular operators can be found, for example, in \cite{Cegielski2015b}.

We have the following relation between weakly and boundedly regular operators:

\begin{proposition}\label{th:WRandBR}
Let $U:\mathcal{H}\rightarrow \mathcal{H}$ be quasi-nonexpansive. Then the following assertions hold:
\begin{enumerate}[(i)]
\item If $U$ is boundedly regular, then $U$ is weakly regular;
\item If $\dim \mathcal{H}<\infty $ and $U$ is weakly regular, then $U$ is boundedly regular.
\end{enumerate}
\end{proposition}

\begin{proof}
See \cite[Proposition 4.1]{CegielskiZalas2014}.
\end{proof}

It is worth mentioning that in a general Hilbert space the weak regularity of $U$ is only a necessary condition for implication \eqref{eq:def:BR} and even a firmly nonexpansive mapping may not have this property; see either \cite[Example 2.9]{Zalas2014} or \cite[Examples 2.14 and 2.15]{GibaliReichZalas2015}. It is also worth mentioning that Cegielski \cite[Definition 4.4]{Cegielski2015} has recently considered a demi-closednss condition referring to a family of operators instead of a single one.

\begin{example}\label{ex:PCisAC}
Let $C\subseteq\mathcal H$ be closed and convex. Then the metric projection $P_C$ satisfies the relation $d(x,C)=\|P_Cx-x\|$ for every $x\in\mathcal H$. Moreover, $\fix P_C=C$. Therefore $P_C$ is boundedly regular and, by Proposition \ref{th:WRandBR}, it is also weakly regular.
\end{example}

\begin{example} \label{ex:NEandWR}
Let $U\colon\mathcal H\rightarrow\mathcal H$ be nonexpansive and assume that $\fix U\neq \emptyset$. Then $U$ is weakly regular; see \cite[Lemma 2]{Opial1967}. Moreover, if $\mathcal H=\mathbb R^n$, then $U$ is boundedly regular.
\end{example}

\begin{example} \label{ex:SubProjandDCPrinciple}
Let $f\colon\mathcal H\rightarrow\mathcal H$ and $P_f$ be as in Example \ref{ex:def:SubPro}. If $f$ is Lipschitz continuous on bounded sets, then $P_f$ is weakly regular; see, for instance, \cite[Theorem 4.2.7]{Cegielski2012}. Note that by \cite[Proposition 7.8]{BauschkeBorwein1996}, we can equivalently assume that $f$ maps bounded sets onto bounded sets or that the subdifferential of $f$ is nonempty and uniformly bounded on bounded sets. Moreover, if $\mathcal H=\mathbb R^n$, then $P_f$ satisfies all the above-mentioned conditions and consequently, by Proposition \ref{th:WRandBR}, $P_f$ is boundedly regular.
\end{example}

\subsection{Regularity of sets}\label{sec:regSets}
Let $C_{i}\subseteq \mathcal{H}$, $i\in I$, be closed and convex sets with a nonempty intersection $C$. Following Bauschke \cite[Definition 2.1]{Bauschke1995}, we propose the following definition.
\begin{definition}\label{def:BR}
 We say that the family $\mathcal C:=\{C_{i}\mid i\in I\}$ is \textit{boundedly regular} if for any bounded sequence $\{x^k\}_{k=0}^{\infty }\subseteq \mathcal H$, the following implication holds:
\begin{equation}
\lim_{k\rightarrow\infty}\max_{i\in I}d(x^k,C_i)=0\quad\Longrightarrow\quad \lim_{k\rightarrow\infty}d(x^k,C)=0.
\end{equation}
\end{definition}

\begin{theorem}
\label{th:bdreg}If at least one of the following conditions is satisfied:
\begin{enumerate}[(i)]
\item $\dim \mathcal{H}<\infty $,
\item $\interior \bigcap_{i\in I}C_{i}\neq \emptyset$,
\item each $C_{i}$ is a half-space,
\end{enumerate}
\noindent then the family $\mathcal{C}:=\{C_{i}\mid i\in I\}$ is boundedly
regular.
\end{theorem}

\begin{proof}
See \cite[Fact 2.2]{Bauschke1995}.
\end{proof}
For more properties of boundedly regular families of sets we refer the reader to Bauschke's PhD thesis \cite{Bauschke1996PhD} and to the review paper \cite{BauschkeBorwein1996}.

\subsection{Hybrid steepest descent method}\label{sec:HSD}
In this section we present a general convergence theorem for the hybrid steepest descent method; compare with (\ref{eq:def:HSD}). In Section \ref{sec:OAM} we apply this theorem to a family of relaxed metric projections $R_k:=\id+\alpha_k(P_{H_k}-\id)$, which constitute our outer-approximation method. Before all this we recall the following technical lemma.

\begin{lemma} \label{th:EquivImp}
Let $\{s_k\}_{k=0}^\infty, \{d_k\}_{k=0}^{\infty
}\subseteq[0,\infty)$ be given. The following conditions are equivalent:

\begin{enumerate}
[(i)]

\item for any $\{n_k\}_{k=0}^\infty\subseteq\{k\}_{k=0}^\infty$, the
following implication holds:
\begin{equation}
\label{eq:th:EquivImp:Cond1}\lim_{k\rightarrow\infty} s_{n_k}=0\quad
\Longrightarrow\quad\lim_{k\rightarrow\infty} d_{n_k}=0;
\end{equation}

\item for any $\{n_k\}_{k=0}^\infty\subseteq\{k\}_{k=0}^\infty$ and for
any $\varepsilon>0$, there are $k_0\geq0$ and $\delta>0$ such that for any
$k\geq k_0$, the following implication holds:
\begin{equation}
\label{eq:th:EquivImp:Cond2}s_{n_k}<\delta\quad\Longrightarrow\quad
d_{n_k}<\varepsilon.
\end{equation}

\end{enumerate}
\end{lemma}

\begin{proof}
See either \cite[Lemma 3.15]{Zalas2014} or \cite[Lemma 2.15]{GibaliReichZalas2015}.
\end{proof}

\begin{theorem} \label{th:conv:HSD}
Let $F\colon\mathcal H\rightarrow\mathcal H$ be $L$-Lipschitz continuous and $\alpha$-strongly monotone, and let $C$ be closed and convex. Moreover, for each $k=0,1,2,\ldots$, let $R_k\colon\mathcal H\rightarrow\mathcal H$ be $\rho_k$-SQNE such that $C\subseteq\fix R_k$ and let $\lambda_k \in[0,\infty)$. Consider the following hybrid steepest descent method:
\begin{equation}\label{eq:th:conv:HSD:def}
z^0\in\mathcal H;\qquad z^{k+1}:=R_kz^k-\lambda_kFR_kz^k.
\end{equation}

Then the sequence $\{z^k\}_{k=0}^\infty$ is bounded. Moreover, if $\rho:=\inf_k\rho_k>0$, $\lim_{k\rightarrow \infty}\lambda_k=0$ and there is an integer $s\geq 1$ such that the implication
\begin{equation} \label{eq:th:conv:HSD:RkRegularity}%
\lim_{k\rightarrow\infty}\sum_{l=0}^{s-1}\| R_{n_k-l} z^{n_k%
-l} -z^{n_k-l}\|=0\quad\Longrightarrow\quad\lim_{k\rightarrow
\infty}d(z^{n_k},C)=0
\end{equation}
holds true for each subsequence $\{n_k\}_{k=0}^\infty \subseteq\{k\}_{k=0}^\infty$, then
$\lim_{k\rightarrow\infty}d(z^k,C)=0$. If, in addition, $\sum_{k=0}^\infty\lambda_k=\infty$, then the sequence $\{z^k\}_{k=0}^\infty$ converges in norm to the unique solution of VI($F$, $C$).
\end{theorem}

The proof of this theorem can be found in \cite[Theorem 3.16]{Zalas2014}. We include it below for the convenience of the reader.

\begin{proof}
Boundedness of $\{z^k\}_{k=0}^\infty$ follows from \cite[Lemma 9]{CegielskiZalas2013a}.

To prove that $d(z^k,C)\rightarrow 0$ it suffices to show, by \cite[Theorem 12]{CegielskiZalas2013a}, that for any $\varepsilon>0$, there are $k_0\geq0$ and $\delta>0$ such that for any $k\geq k_0$, the following implication holds:
\begin{equation} \label{eq:th:conv:HSD:proof:1}
\sum_{l=0}^{s-1}\rho_{k-l}\|R_{k-l}z^{k-l}-z^{k-l}\|^2<\delta\quad\Longrightarrow\quad
d^2(z^k,C)<\varepsilon.
\end{equation}
To this end, we define for each $k=0,1,2,\ldots,$
\begin{equation} \label{eq:th:conv:HSD:proof:2}
s_k:=\sum_{l=0}^{s-1}\rho_{k-l}\|R_{k-l}z^{k-l}-z^{k-l}\|^2
\end{equation}
and
\begin{equation} \label{eq:th:conv:HSD:proof:3}
d_k:=d^2(z^k,C).
\end{equation}
It is easy to see that, by assumption, $\rho >0$ and by (\ref{eq:th:conv:HSD:RkRegularity}), for any subsequence $\{n_k\}_{k=0}^\infty\subseteq\{k\}_{k=0}^\infty$, the
following implication holds:
\begin{equation} \label{eq:th:conv:HSD:proof:4}
\lim_{k\rightarrow\infty} s_{n_k}=0\quad
\Longrightarrow\quad\lim_{k\rightarrow\infty} d_{n_k}=0.
\end{equation}
Clearly, Lemma \ref{th:EquivImp} ((i)$\Rightarrow$(ii)) with $n_k\leftarrow k$ shows that implication (\ref{eq:th:conv:HSD:proof:1}) holds. Therefore, by \cite[Theorem 12 (i)]{CegielskiZalas2013a}, we get $\lim_{k\rightarrow\infty}d(z^k,C)=0$.

To finish the proof, note that the assumption $\sum_{k=0}^\infty\lambda_k=\infty$, when combined with \cite[Theorem 12 (v)]{CegielskiZalas2013a}, yields the assertion.
\end{proof}

\section{Outer approximation method}\label{sec:OAM}

\begin{theorem} \label{th:main}
Let $F\colon\mathcal H\rightarrow\mathcal H$ be $L$-Lipschitz continuous and $\alpha$-strongly monotone, and let $C\subset\mathcal H$ be closed and convex. Moreover, for each $k=0,1,2,\ldots$, let $T_k\colon\mathcal H\rightarrow\mathcal H$ be a cutter such that $C\subseteq\fix T_k$ and let $\lambda_k \in[0,\infty)$. Consider the following outer approximation method:
\begin{equation}\label{eq:def:xk}
x^0\in\mathcal H;\qquad x^{k+1}:=R_k(x^k-\lambda_k F x^k), \text{ for } k=0,1,2,\ldots,
\end{equation}
where
\begin{equation}\label{eq:def:Rk}
R_k:=\id+\alpha_k(P_{H_k}-\id),
\end{equation}
\begin{equation}\label{eq:def:Hk}
H_k:=\{z\in\mathcal H\mid\langle z-T_k x^k
,x^k-T_k x^k \rangle\leq 0\}
\end{equation}
and where $\alpha_k\in[\varepsilon,2-\varepsilon]$ is the user-chosen relaxation parameter for some $\varepsilon >0$.

Then the sequence $\{x^k\}_{k=0}^\infty$ is bounded. Moreover, if $\lim_{k\rightarrow \infty}\lambda_k=0$ and there is an integer $s\geq 1$ such that the implication
\begin{equation} \label{eq:Cond:TkRegularity}%
\lim_{k\rightarrow\infty}\sum_{l=0}^{s-1}\| T_{n_k-l} x^{n_k%
-l} -x^{n_k-l}\|=0\quad\Longrightarrow\quad\lim_{k\rightarrow
\infty}d(x^{n_k},C)=0
\end{equation}
holds true for each subsequence $\{n_k\}_{k=0}^\infty \subseteq\{k\}_{k=0}^\infty$, then
$\lim_{k\rightarrow\infty}d(x^k,C)=0$. If, in addition, $\sum_{k=0}^\infty\lambda_k=\infty$, then the sequence $\{x^k\}_{k=0}^\infty$ converges in norm to the unique solution of VI($F$, $C$).
\end{theorem}

The detailed proof of this theorem is given in Subsection \ref{sec:OAM:ConvAnal}. For now we discuss only basic properties of this method with a geometric interpretation and a simple motivation for the convergence conditions.

As we have already mentioned in the Introduction, Lipschitz continuity and strong monotonicity of $F$ guarantee the existence and uniqueness of a solution to VI($F$, $C$).

Observe that for every $k=0,1,2,\ldots,$ the set $H_k$ is a half-space unless $x^k=T_kx^k$ in which case it is all of $\mathcal H$. Therefore
(\ref{eq:def:xk}) has the following explicit form:
\begin{equation} \label{eq:computationRk}
x^{k+1}=R_kz^k=
\begin{cases}
z^k-\alpha_k\frac{\left\langle z^k-T_k x^k
,\ x^k-T_k x^k \right\rangle }{\| x^k-T_k
x^k \|^2}(x^k-T_k x^k ) & \text{if }%
z^k\notin H_k,\\
z^k & \text{if }z^k\in H_k,
\end{cases}
\end{equation}
where
\begin{equation} \label{eq:def:zk}
z^k:=x^k-\lambda_kFx^k.
\end{equation}
Moreover, since for every $k=0,1,2,\ldots$, the operator $T_k$ is a cutter such that $C\subseteq \fix  T_k $, then the subset $C$ is outerly approximated by $H_k$. Moreover, $C\subseteq \fix  T_k \subseteq H_k$. We illustrate the iterative method (\ref{eq:def:xk}) in this case in Figure \ref{fig:algorithm}.

\begin{figure}[htb]
\centering
\includegraphics[scale=0.25]{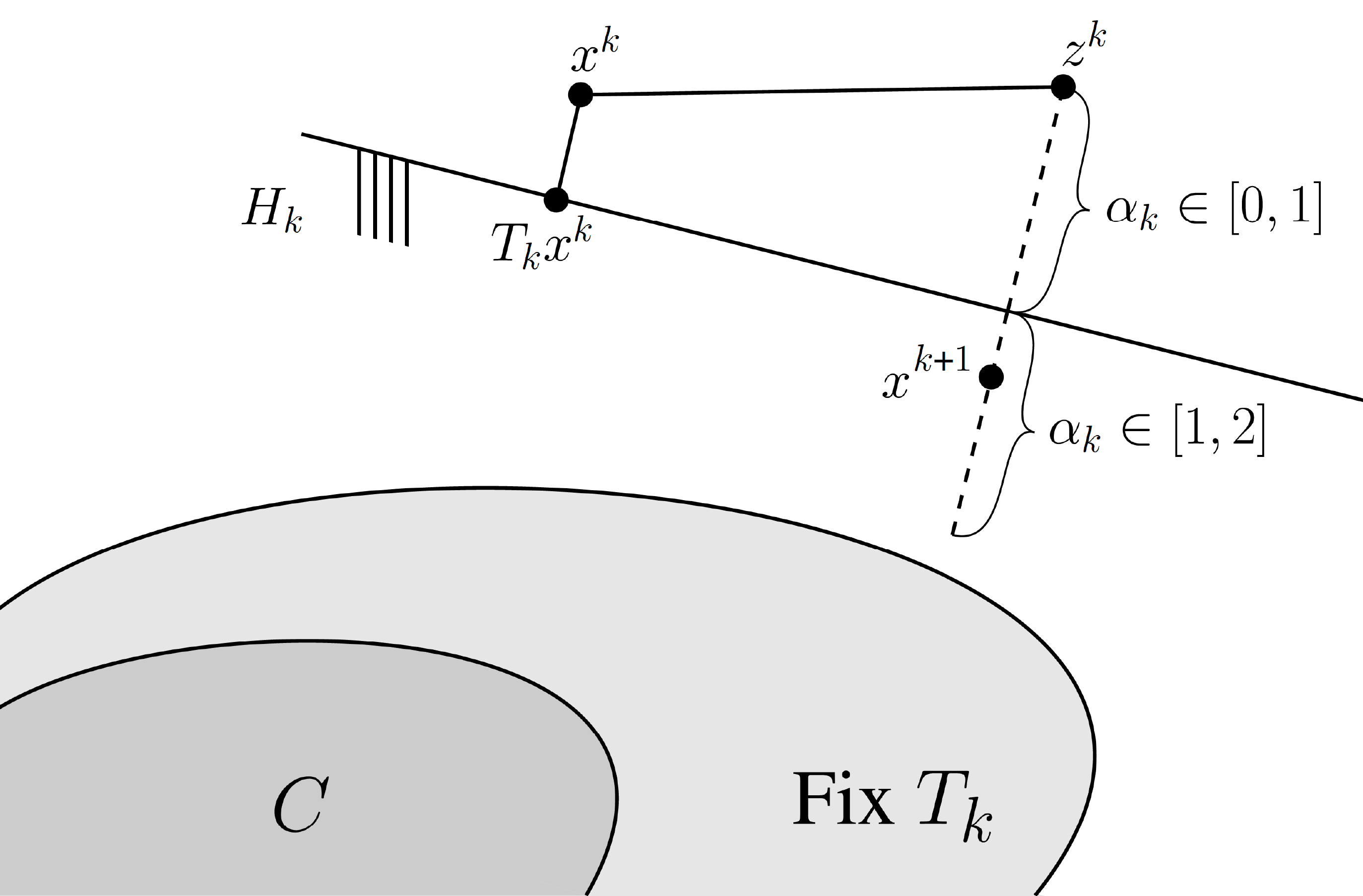}
\caption{Illustration of the iterative step of the outer approximation method (\ref{eq:def:xk}).}
\label{fig:algorithm}
\end{figure}

By imposing conditions on $\{\lambda_k\}_{k=0}^\infty$, following Fukushima \cite{Fukushima1986} and others, see \cite{CensorGibali2008}, \cite{CegielskiGibaliReichZalas2013} and \cite{GibaliReichZalas2015}, we replace the fixed step size $\lambda$ in the gradient projection method (\ref{eq:def:GP}) by a null, non-summable sequence. This condition is quite common in optimization theory and, in particular, appears in the context of the HSD method (\ref{eq:th:conv:HSD:def}); see, for example, the papers by Yamada and Ogura \cite{YamadaOgura2004}, Hirstoaga \cite{Hirstoaga2006}, Aoyama and Kohsaka \cite{AoyamaKohsaka2014}, Cegielski and Zalas \cite{CegielskiZalas2013a, CegielskiZalas2014} and Cegielski and Al-Musallam \cite{CegielskiMusallam2016}. In many cases, the choice of the sequence $\{\lambda_k\}_{k=0}^\infty$ is more restrictive than the one proposed in Theorem \ref{th:main}. Examples of such restrictions can be found in the papers by Halpern \cite{Halpern1967}, Lions \cite{Lions1977}, Wittmann \cite{Wittmann1992}, Bauschke \cite{Bauschke1996}, Deutsch and Yamada \cite{DeutschYamada1998}, Yamada \cite{Yamada2001}, Xu and Kim \cite{XuKim2003}, Zeng \textit{et al.} \cite{ZengWongYao2007}, Takahashi and Yamada \cite{TakahashiYamada2008}, Aoyama and Kimura \cite{AoyamaKimura2011}, and Zhang and He \cite{ZhangHe2013}.

The condition \eqref{eq:Cond:TkRegularity} is essential as we now explain in detail. Although this condition imposes some regularity on the sequence $\{x^k\}_{k=0}^\infty$, the convergence of which is under investigation, this does not reduce its generality. One could, for example, assume a variant of \eqref{eq:Cond:TkRegularity}, where instead of the trajectory of method \eqref{eq:def:xk}, any arbitrary sequence is used, as in \cite{TakahashiTakeuchiKubota2008, AoyamaKohsaka2014}. This, however, could be more difficult to verify, since an arbitrary sequence $\{x^k\}_{k=0}^\infty$ does not provide any additional information concerning its structure. By restricting our attention to trajectories generated by the outer approximation method \eqref{eq:def:xk}, we are able to utilize such a structure and therefore the verification of \eqref{eq:Cond:TkRegularity} should require at most as much effort as its verification for any arbitrary sequence.
Notice, however, that \eqref{eq:Cond:TkRegularity} refers not only to the sequence $\{x^k\}_{k=0}^\infty$, but first of all to the sequence of operators $\{T_k\}_{k=0}^\infty$, which determine method \eqref{eq:def:xk}. We give now several simple examples where this condition is satisfied. To make the introductory analysis simpler we assume that $s=1$, $T_k=T$ and $C=\fix T$ for some cutter $T\colon\mathcal H\rightarrow\mathcal H$ and for every $k=0,1,2,\ldots$. It is not difficult to see that if $T$ is boundedly regular and, in particular, if $T=P_C$, then \eqref{eq:Cond:TkRegularity} holds true. Now assume that $\mathcal H=\mathbb R^n$. If $T$ is weakly regular and, in particular, when $T$ is either nonexpansive or $T=P_f$ (see Examples \ref{ex:NEandWR} and \ref{ex:SubProjandDCPrinciple}), then again \eqref{eq:Cond:TkRegularity} is satisfied.

The parameter $s>1$ in \eqref{eq:Cond:TkRegularity} enables us to use $s$-almost cyclic and $s$-intermittent controls. Examples for this case are presented in Section \ref{sec:VIoverCFPP}.

Historically, condition \eqref{eq:Cond:TkRegularity} in this form appeared in Zalas' PhD thesis \cite[Theorem 3.16]{Zalas2014} in the context of the HSD method, and more recently, in \cite{GibaliReichZalas2015} in connection with the outer approximation method. Some of its weaker forms can be found in \cite[Definition 19]{CegielskiZalas2013a}, \cite[Definition 6.1]{CegielskiZalas2014} and \cite[Definition 4.1]{Cegielski2014}. Similar regularity conditions were proposed by many authors; see, for example, Bauschke \textit{et al.} \cite[Definition 3.7, Definition 4.8]{BauschkeBorwein1996}, Yamada \textit{et al.} \cite[Definition 1]{YamadaOgura2004}, Hirstoaga \cite[Condition 2.2(iii)]{Hirstoaga2006}, Takahashi \textit{et al.} \cite[NST condition (I) ]{TakahashiTakeuchiKubota2008}, Cegielski \cite[Theorem 3.6.2, Definition 5.8.5]{Cegielski2012}, and Aoyama \textit{et al.} \cite[Condition (Z)]{AoyamaKohsaka2014}.

\subsection{Convergence analysis} \label{sec:OAM:ConvAnal}
We begin this section with several simple observations. Let $\{x^k\}_{k=0}^\infty$ and $\{z^k\}_{k=0}^ \infty$ be two sequences generated by  (\ref{eq:def:xk}) and (\ref{eq:def:zk}), respectively. Then for any $k=0,1,2,\ldots$, the vector $z^{k+1}$ depends recursively on $z^k$ via the formula
\begin{equation} \label{eq:zkxkDependance}
z^0:=x^0-\lambda_0Fx^0,
\qquad
z^{k+1}=R_k z^k -\lambda_{k+1}FR_k z^k.
\end{equation}

According to Example \ref{ex:MetProj}, for each $k=0,1,2,\ldots$, the operator $R_k$ defined by (\ref{eq:def:Rk}) is $\rho_k$-SQNE with $\rho_k:=\frac{2-\alpha_k}{\alpha_k}$. Moreover, $\rho:=\frac{\varepsilon}{2-\varepsilon}$ satisfies the inequalities $0<\rho\leq \inf_k\rho_k$. Furthermore, for each $k=0,1,2,\ldots,$ we get $C\subseteq \fix R_k$. Consequently, one may apply Theorem \ref{th:conv:HSD} to the sequences $\{z^k\}_{k=0}^\infty$ and $\{R_k\}_{k=0}^\infty$, which is the key idea in our convergence analysis. We begin with the following lemma:

\begin{lemma} \label{th:zkxk}
The following statements hold true:
\begin{enumerate}[(i)]
  \item The sequences $\{x^k\}_{k=0}^\infty$ and $\{z^k\}_{k=0}^\infty$ are bounded;
  \item For any subsequence $\{n_k\}_{k=0}^\infty \subseteq\{k\}_{k=0}^\infty$, we have
  \begin{align} \label{eq:th:zkxk:equiv}
    \lim_{k\rightarrow\infty}(R_{n_k} z^{n_k} -z^{n_k})=0
    & \  \Longleftrightarrow \  \lim_{k\rightarrow\infty}(T_{n_k} x^{n_k} -x^{n_k})=0\\
    & \  \Longleftrightarrow \  \lim_{k\rightarrow\infty} (x^{n_k+1}-x^{n_k})=0 \\
    &\  \Longleftrightarrow \  \lim_{k\rightarrow\infty} (z^{n_k+1}-z^{n_k})=0;
  \end{align}
  \item If $\lim_{k\rightarrow\infty} d(x^{n_k},C)=0$, then all limits in (ii) are equal to zero;
  \item $\lim_{k\rightarrow\infty} d(x^{n_k},C)=0 $ if and only if $\lim_{k\rightarrow\infty} d(z^{n_k},C)=0$;
  \item $\lim_{k\rightarrow\infty} x^{n_k}=\lim_{k\rightarrow\infty} z^{n_k}$ if at least one of the limits exists.
\end{enumerate}
\end{lemma}

\begin{proof}
First we show that (i) holds true. Note that Theorem \ref{th:conv:HSD} (i), when combined with (\ref{eq:zkxkDependance}), leads to the boundedness of the sequence $\{z^k\}_{k=0}^\infty$. To show that $\{x^k\}_{k=0}^\infty$ is also bounded, fix $z\in C$. By the definition of $x^{k+1}$ (see (\ref{eq:def:xk})), the quasi-nonexpansivity of $R_k$ and the inclusion $C\subseteq \fix R_k$, it is easy to see that
\begin{equation}
\|x^{k+1}-z\|=\|R_kz^k-z\|\leq \|z^k-z\|
\end{equation}
for each $k=0,1,2,\ldots$. Therefore $\{x^k\}_{k=0}^\infty$ is also bounded, as asserted.

Now we proceed to statement (ii). By (i), the sequence $\{x^k\}_{k=0}^\infty$ is bounded. Therefore, by the Lipschitz continuity of $F$, there is $M>0$ such that for each $k=0,1,2\ldots$, we have $\|Fx^k\|\leq M$. Let $\{n_k\}_{k=0}^\infty\subseteq\{k\}_{k=0}^\infty$. Assume that $\|T_{n_k}x^{n_k}-x^{n_k}\|\neq 0$. Then, by (\ref{eq:computationRk}), the Cauchy-Schwarz
inequality, (\ref{eq:def:zk}) and the triangle inequality,
\begin{align}\label{eq:th:zkxk:equiv:proof:1}
\| R_{n_k} z^{n_k} -z^{n_k}\| & =\alpha_{n_k}
\left\| \frac{\langle z^{n_k}-T_{n_k} x^{n_k}, x^{n_k}-T_{n_k} x^{n_k} \rangle} {\| x^{n_k}-T_{n_k} x^{n_k} \|^2}(x^{n_k}-T_{n_k}
x^{n_k} )\right\| \nonumber\\
& \leq 2\| z^{n_k}-T_{n_k} x^{n_k} \|= 2\left\|
x^{n_k}-\lambda_{n_k} F x^{n_k}
-T_{n_k} x^{n_k%
} \right\| \nonumber\\
& \leq 2\left(\| T_{n_k}x^{n_k}-x^{n_k}\|+\lambda_{n_k}M\right).
\end{align}
Moreover, if $\|T_{n_k}x^{n_k}-x^{n_k}\|= 0$, then $R_{n_k}=\id$ and inequality (\ref{eq:th:zkxk:equiv:proof:1}) holds trivially in this case.

Observe that for each $k=0,1,2,\ldots$, we have $T_k x^k =P_{H_k}
x^k $. Moreover, $R_k=\id +\alpha_k(P_{H_k}-\id )$ is NE, since $P_{H_k}$ is FNE  and $\alpha_k\in\lbrack\varepsilon,2-\varepsilon]$ (see Example \ref{ex:MetProj}). Hence, by the triangle inequality, (\ref{eq:def:xk}) and (\ref{eq:def:zk}), we have
\begin{align}\nonumber\label{eq:th:zkxk:equiv:proof:2}
\| T_{n_k} x^{n_k} -x^{n_k}\| & = \frac{1}{\alpha_k}
\left\| \alpha_k \left(P_{H_{n_k}} x^{n_k} -x^{n_k}\right)\right\|
\leq \frac 1\varepsilon\| R_{n_k} x^{n_k} -x^{n_k}\| \\\nonumber
& \leq \frac 1\varepsilon \left(\| R_{n_k} x^{n_k} -x^{n_k+1}\|
+\|x^{n_k+1}-x^{n_k}\|\right) \\ \nonumber
& =\frac 1\varepsilon \left(\|R_{n_k}x^{n_k}-R_{n_k}z^{n_k}\|+\|x^{n_k+1}-x^{n_k}\| \right)\\
& \leq \frac 1\varepsilon \left(\lambda_{n_k}M+ \| x^{n_k+1}-x^{n_k}\|\right).
\end{align}
Using (\ref{eq:def:zk}) for $x^{n_k+1}$ and $x^{n_k}$ together with the
triangle inequality, we obtain
\begin{equation}
\| x^{n_k+1}-x^{n_k}\|\leq\| z^{n_k+1}-z^{n_k}\|
+M(\lambda_{n_k+1}+\lambda_{n_k}).
\label{eq:th:zkxk:equiv:proof:3}%
\end{equation}
Moreover, by (\ref{eq:zkxkDependance}) applied to $z^{n_k+1}$ and the triangle inequality,
\begin{equation}
\| z^{n_k+1}-z^{n_k}\|\leq\| R_{n_k} z^{n_k}
-z^{n_k}\|+\lambda_{n_k+1}M. \label{eq:th:zkxk:equiv:proof:4}%
\end{equation}
The assumption that $\lambda_k\rightarrow0$, when combined with inequalities
(\ref{eq:th:zkxk:equiv:proof:1})--(\ref{eq:th:zkxk:equiv:proof:4}), yields the
equivalence \eqref{eq:th:zkxk:equiv}.

Now we show that (iii) holds true. To this end, assume that
\begin{equation}
\lim_{k\rightarrow\infty}d(x^{n_k},C)=0. \label{eq:th:zkxk:equiv:proof:5}%
\end{equation}
We claim that
\begin{equation}
\lim_{k\rightarrow\infty}\| x^{n_k+1}-x^{n_k}\|=0.
\end{equation}
Indeed, by the triangle inequality, (\ref{eq:def:xk}) and by the nonexpansivity of $R_{n_k}$, we obtain for all $k\geq 0$,
\begin{align}
\| x^{n_k+1}-x^{n_k}\| & \leq\| x^{n_k+1}-R_{n_k}
x^{n_k} \|+\| R_{n_k} x^{n_k} -x^{n_k%
}\|\nonumber\\
& =\| R_{n_k} z^{n_k} -R_{n_k} x^{n_k%
} \|+\alpha_{n_k} d(x^{n_k},H_{n_k})\nonumber\\
& \leq\| z^{n_k}-x^{n_k}\|+ 2\ d(x^{n_k},H_{n_k})\nonumber\\
& \leq \lambda_{n_k}M+2\ d(x^{n_k},H_{n_k}),
\end{align}
where the last equality follows from
\begin{equation}\label{eq:th:zkxk:equiv:proof:6}
\lambda_{n_k}M\geq\|z^{n_k}-x^{n_k}\|.
\end{equation}
Since for all $k\geq0$, $C\subseteq H_{n_k}$, we have
\begin{equation}
d(x^{n_k},H_{n_k})\leq d(x^{n_k},C).
\end{equation}
Thus
\begin{equation}
\| x^{n_k+1}-x^{n_k}\|\leq\lambda_{n_k}M+2\ d(x^{n_k},C)
\end{equation}
and the right-hand side of the above inequality converges to zero by
(\ref{eq:th:zkxk:equiv:proof:5}) and the assumption that $\lambda_k\rightarrow 0$.

Statements (iv) and (v) follow directly from (\ref{eq:th:zkxk:equiv:proof:6}) and again by the assumption that $\lambda_k\rightarrow 0$. This completes the proof.
\end{proof}

\begin{proof}[Proof of Theorem \ref{th:main}]
Let $\{z^k\}_{k=0}^\infty$ be the sequence defined in (\ref{eq:def:zk}), corresponding to $\{x^k\}_{k=0}^\infty$. Moreover, let $x^*$ be the unique solution of VI($F$, $C$). We show that $\{z^k\}_{k=0}^\infty$ converges in norm to $x^*$, which in view of Lemma \ref{th:zkxk} (v), yields the result. To this purpose, it suffices, by Theorem \ref{th:conv:HSD} and (\ref{eq:zkxkDependance}), to show that there is an integer $s\geq 1$ such that the implication
\begin{equation} \label{eq:th:main:proof:regularityRk}%
\lim_{k\rightarrow\infty}\sum_{l=0}^{s-1}\| R_{n_k-l} z^{n_k%
-l} -z^{n_k-l}\|=0\quad\Longrightarrow\quad\lim_{k\rightarrow
\infty}d(z^{n_k},C)=0
\end{equation}
holds true for each subsequence $\{n_k\}_{k=0}^\infty \subseteq\{k\}_{k=0}^\infty$. Let $\{n_k\}_{k=0}^\infty \subseteq\{k\}_{k=0}^\infty$ and assume that the antecedent of (\ref{eq:th:main:proof:regularityRk}) holds true, that is,
\begin{equation} \label{eq:th:main:proof:1}%
\lim_{k\rightarrow\infty}\sum_{l=0}^{s-1}\| R_{n_k-l} z^{n_k%
-l} -z^{n_k-l}\|=0.
\end{equation}
Thus, using Lemma \ref{th:zkxk} (ii), we arrive at
\begin{equation} \label{eq:th:main:proof:1}%
\lim_{k\rightarrow\infty}\sum_{l=0}^{s-1}\| T_{n_k-l} x^{n_k%
-l} -x^{n_k-l}\|=0.
\end{equation}
By \eqref{eq:Cond:TkRegularity}, we get
\begin{equation} \label{eq:th:main:proof:2}%
\lim_{k\rightarrow\infty}d( x^{n_k},C)=0,
\end{equation}
which, when combined with Lemma \ref{th:zkxk} (iv), yields
\begin{equation} \label{eq:th:main:proof:3}%
\lim_{k\rightarrow\infty}d( z^{n_k},C)=0.
\end{equation}
This completes the proof.
\end{proof}

\section{VIs over the common fixed point set}\label{sec:VIoverCFPP}
In this section we assume that $C:=\bigcap_{i\in I}C_i$, where each $C_i\subseteq\mathcal H$ is closed and convex and $I:=\{1,\ldots,m\}$.
\begin{theorem} \label{th:main2}
Let $F\colon\mathcal H\rightarrow\mathcal H$ be $L$-Lipschitz continuous and $\alpha$-strongly monotone, and assume that for each $i\in I$, we have
$ C_i=\fix U_i$ for some cutter operator $U_i\colon\mathcal H\rightarrow \mathcal H$. Let the sequence $\{x^k\}_{k=0}^\infty$ be defined by the outer approximation method \eqref{eq:def:xk}--\eqref{eq:def:Hk} and for every $k=0,1,2,\ldots,$ let $T_k\colon\mathcal H\rightarrow\mathcal H$ be defined either by a simultaneous
\begin{equation}\label{eq:def:convex}
T_k:=\sum_{i\in I_k}\omega_{i}^kU_{i}
\end{equation}
or a composition
\begin{equation}\label{eq:def:composition}
T_k:=\frac 1 2(\id+\prod_{i\in I_k}U_{i}),
\end{equation}
algorithmic operator, where $I_k\subseteq I$, $\sum_{i\in I_k}\omega_i^k=1$ and $0<\varepsilon\leq \omega_i^k\leq 1$.

Then the sequence $\{x^k\}_{k=0}^\infty$ is bounded. Moreover, if $\lim_{k\rightarrow \infty}\lambda_k=0$, $U_i$ is boundedly regular for every $i\in I$, $\{C_i \mid i\in I\}$ is boundedly regular and there is an integer $s\geq 1$ such that $I= I_{k-s+1}\cup\ldots\cup I_k$ for all $k\geq s$, then $\lim_{k\rightarrow\infty} d(x^k,C)=0$. If, in addition, $\sum_{k=0}^\infty\lambda_k=\infty$, then the sequence $\{x^k\}_{k=0}^\infty$ converges in norm to the unique solution of VI($F$, $C$).
\end{theorem}

\begin{proof}
We begin with several simple observations. The first one is that for each $k=0,1,2,\ldots,$ the operator $T_k$ defined by either \eqref{eq:def:convex} or \eqref{eq:def:composition} is a cutter such that $C\subseteq \fix T_k=\bigcap_{i\in I_k} \fix U_i$. This follows from Theorem \ref{th:SQNE} and Corollary \ref{th:cuttersAndQNE:corollary}. Therefore it is reasonable to consider an outer approximation method with these particular algorithmic operators $T_k$. Consequently, by Theorem \ref{th:main}, $\{x^k\}_{k=0}^\infty$ is bounded.

Another observation is that, by \cite[Lemma 3.5]{ReichZalas2015}, for each subsequence $\{n_k\}_{k=0}^\infty \subseteq\{k\}_{k=0}^\infty$ and $l\in \{0,\ldots,s-1\}$, we have

\begin{equation}\label{eq:th:main2:proof:1}
\lim_{k\rightarrow\infty}\| T_{n_k-l} x^{n_k%
-l} -x^{n_k-l}\|=0  \quad \Longrightarrow \quad
\lim_{k\rightarrow\infty} \max_{i\in I_{n_k-l}} d(x^{n_k-l}, \fix U_i)=0.
\end{equation}
In order to complete the proof, in view of Theorem \ref{th:main} and the bounded regularity of $\{C_i \mid i\in I\}$, it suffices to show that
\begin{equation}\label{eq:th:main2:proof:step2:1}
\lim_{k\rightarrow\infty}\sum_{l=0}^{s-1}\| T_{n_k-l} x^{n_k%
-l} -x^{n_k-l}\|=0 \quad \Longrightarrow \quad
\lim_{k\rightarrow\infty} \max_{i\in I} d(x^{n_k}, \fix U_i)=0.
\end{equation}
Indeed, assume that the left-hand side of (\ref{eq:th:main2:proof:step2:1}) holds, that is,
\begin{equation}\label{eq:th:main2:proof:step2:2}
\lim_{k\rightarrow\infty}\sum_{l=0}^{s-1}\| T_{n_k-l} x^{n_k%
-l} -x^{n_k-l}\|=0,
\end{equation}
which, by \eqref{eq:th:main2:proof:1}, implies that for each $l=0,1,\ldots,s-1$,
\begin{equation}\label{eq:th:main2:proof:step2:3}
\lim_{k\rightarrow\infty} \max_{j\in I_{n_k-l}} d(x^{n_k-l}, \fix U_j)=0.
\end{equation}
Again by (\ref{eq:th:main2:proof:step2:2}), the triangle inequality and Lemma \ref{th:zkxk} (ii) applied to $n_k\leftarrow(n_k-l)$, for every $l=1,2,\ldots,s-1$, we get
\begin{equation} \label{eq:th:main2:proof:step2:4}
\lim_{k\rightarrow\infty}\|x^{n_k}-x^{n_k-l}\|=0.
\end{equation}

Let $i\in I$. The control $\{I_k\}_{k=0}^\infty$ satisfies $I=I_{n_k}\cup I_{n_k-1}\cup\ldots I_{n_k -s+1}$ for all $k\geq 0$. Consequently, for each $k\geq s-1$, there is $l_k\in\{0,\ldots,s-1\}$ such that $i\in I_{n_k-l_k}$. By the definition of the metric projection and the triangle inequality, we have
\begin{align}\label{eq:th:main2:proof:step2:6}\nonumber
d(x^{n_k}, \fix U_i) &=\|P_{\fix U_i}x^{n_k}-x^{n_k}\|\leq \|P_{\fix U_i}x^{n_k-l_k}-x^{n_k}\|\\ \nonumber
&\leq \|P_{\fix U_i}x^{n_k-l_k}-x^{n_k-l_k}\|+\|x^{n_k}-x^{n_k-l_k}\|\\
&= d(x^{n_k-l_k}, \fix U_i) +\|x^{n_k}-x^{n_k-l_k}\|.
\end{align}
Therefore (\ref{eq:th:main2:proof:step2:6}), (\ref{eq:th:main2:proof:step2:4}) and (\ref{eq:th:main2:proof:step2:3}) imply that
\begin{equation} \label{eq:th:main2:proof:step2:7}
\lim_{k\rightarrow\infty} \max_{i\in I}d (x^{n_k}, \fix U_i)=0,
\end{equation}
which completes the proof.
\end{proof}

\begin{theorem} \label{th:main3}
Let $F\colon\mathcal H\rightarrow\mathcal H$ be $L$-Lipschitz continuous and $\alpha$-strongly monotone, and assume that for each $i\in I$, we have
$ C_i=\fix U_i=p_i^{-1}(0)$ for some cutter operator $U_i\colon\mathcal H\rightarrow \mathcal H$ and a proximity function $p_i\colon\mathcal H\rightarrow [0,\infty)$. Let the sequence $\{x^k\}_{k=0}^\infty$ be defined by the outer approximation method \eqref{eq:def:xk}--\eqref{eq:def:Hk} and for every $k=0,1,2,\ldots,$ let $T_k\colon\mathcal H\rightarrow\mathcal H$ be defined by the maximum proximity algorithmic operator
\begin{equation}\label{eq:def:maxProx}
T_k:=U_{i_k}, \quad \text{where} \quad i_k=\argmax_{i\in I_k}p_i(x^k),
\end{equation}
and where $I_k\subseteq I$.

Then the sequence $\{x^k\}_{k=0}^\infty$ is bounded. Moreover, if $\lim_{k\rightarrow \infty}\lambda_k=0$, $U_i$ is boundedly regular for every $i\in I$, $\{C_i \mid i\in I\}$ is boundedly regular, there is an integer $s\geq 1$ such that $I= I_{k-s+1}\cup\ldots\cup I_k$ for all $k\geq s$ and for every bounded $\{y^k\}_{k=0}^\infty\subseteq\mathcal H$, we have $\lim_{k\rightarrow\infty}p_i(y^k)=0 \Longleftrightarrow \lim_{k\rightarrow\infty}d(y^k,C_i)=0$, then $\lim_{k\rightarrow\infty} d(x^k,C)=0$. If, in addition, $\sum_{k=0}^\infty\lambda_k=\infty$, then the sequence $\{x^k\}_{k=0}^\infty$ converges in norm to the unique solution of VI($F$, $C$).
\end{theorem}
\begin{proof}
  Note that one can easily see that \eqref{eq:th:main2:proof:1} holds true for $T_k$ defined in \eqref{eq:def:maxProx}. Therefore the proof remains the same as for Theorem \ref{th:main2}.
\end{proof}
This type of the maximum proximity algorithmic operator can be found in \cite{KolobovReichZalas2016} although its projected variant can also be found in \cite[Section 5.8.4.1]{Cegielski2012}.

The relation $C_i=\fix U_i=p_i^{-1}(0)$ becomes clearer once we assume that the computation of $p_i$ is at most as difficult as the evaluation of $U_i$ and this is at most as difficult as projecting onto $C_i$. These assumptions are satisfied if, for example, the set $C_i=\{z\in\mathcal H\mid f_i(z)\leq 0\}$ is a sublevel set of a convex functional $f_i$, the operator $U_i=P_{f_i}$ is a subgradient projection and the proximity $p_i=f^+_i$, where $f_i^+(x):=\max \{0, f_i(x)\}$. The remaining part is to verify whether $p_i(y^k)\rightarrow 0 \Longleftrightarrow d(y^k,C_i)\rightarrow 0$, which we show in the following lemma.

\begin{lemma}\label{th:regularityOfSubPro}
Let $f_i\colon\mathbb R^n\rightarrow \mathbb R$ be convex and assume that $S(f_i,0)\neq\emptyset$, $i\in I$. Moreover, let $\{i_k\}_{k=0}^\infty\subseteq I$. Then for every bounded sequence $\{y^k\}_{k=0}^\infty$, we have
\begin{equation}
  \lim_{k\rightarrow \infty}\|P_{f_{i_k}}y^k-y^k\|=0 \quad \Longleftrightarrow \quad \lim_{k\rightarrow \infty} f_{i_k}^+(y^k)=0 \quad \Longleftrightarrow \quad \lim_{k\rightarrow \infty} d(y^k, S(f_{i_k},0))=0.
\end{equation}

\end{lemma}
\begin{proof}
First, assume that $I=\{i\}$. Observe that the implications ``$\Longrightarrow$'' follow directly from the proof of \cite[Lemma 24]{CegielskiZalas2013a}, which indicates that $P_{f_i}$ is boundedly regular. Note that since $P_{f_i}$ is a cutter, we have  $\|P_{f_i}x-x\|\leq d(x,\fix P_{f_i})$ for every $x\in\mathbb R^n$, where $\fix P_{f_i}=S(f_i,0)$. This shows equivalence for $I=\{i\}$. Now we assume that $I=\{1,\ldots,m\}$. To complete the proof we decompose the set $K=\{0,1,2,\ldots\}$ into subsets $K_i:=\{k\in K \mid i_k=i\}$. After doing this, we can repeat the first argument for every component $K_i$, separately.
\end{proof}

The condition that $I\subseteq I_{k-s+1}\cup\ldots\cup I_k$ for all $k\geq s$ and some $s\geq 1$ appears in the literature as $s$-\textit{intermittent control}, whereas for $|I_k|=1$ it is known as $s$-\textit{almost cyclic}; see, for example \cite[Definition 3.18]{BauschkeBorwein1996}. We comment now on a practical realization of this condition in the context of projection and subgradient projection algorithms.

\begin{example}[Block projection algorithms]\label{ex:metProjOper}
Let $C=\bigcap_{i\in I} C_i$, and set $U_i=P_{C_i}$ and $p_i=d(\cdot,C_i)$. For a fixed block size $1\leq b \leq m$, let $I_0=\{1,\ldots,b\}$ and let $l_k$ be the last index from a given $I_k$.
We define $I_{k+1}=(\{l_k,\ldots,l_k+b-1\}\ \text{ mod } b)+1$. In principle, $I_k$ consists of the next $b$ indices following $l_k$, which in the case of $b$ dividing $m$ is nothing but a cyclic way of changing fixed blocks $I_1,\ldots,I_{m/b}$, each of them of size $b$. We can visualize the definition of $I_k$ in the following way:
\begin{equation}
  \underbrace{1,2,\ldots,b}_{I_0},\ \underbrace{b+1,b+2,\ldots,2b}_{I_1},\ \ldots\ ,
  \underbrace{m-1,m, 1, 2,\ldots,b-2}_{I_k},\ \underbrace{b-1,b,\ \ldots\ }_{I_{k+1}}
\end{equation}
Following Theorems \ref{th:main2} and \ref{th:main3}, we have the following examples of projection algorithmic operators which determine our outer approximation method:
\begin{enumerate}[a)]
  \item cyclic projection operator: $T_k=P_{C_{[k]}}$, where $[k]=(k \text{ mod }m)+1$;
  \item remotest-set projection operator: $T_k:=P_{C_{i_k}}$, where $i_k=\argmax_{i\in I_k}d(x^k,C_i)$; see Theorem \ref{th:main3};
  \item simultaneous projection operator: $T_k:=\frac{1}{|I_k|}\sum_{i\in I_k}P_{C_i}$;
  \item composition projection operator: $T_k:=\frac 1 2 (\id +\prod_{i\in I_k}P_{C_i})$.
\end{enumerate}

\end{example}

\begin{example}[Block subgradient projection algorithms in $\mathbb R^n$]\label{ex:subProjOper}
Let $C=\bigcap_{i\in I} C_i$, where $C_i=\{z\in\mathbb R^n\mid f_i(z)\leq 0\}$ and $f_i\colon\mathbb R^n\rightarrow \mathbb R$ is convex.  We set $U_i=P_{f_i}$, see Example \ref{ex:def:SubPro}, and $p_i=f_i^+$. For a fixed block size $1\leq b \leq m$ we define $I_k$ as in Example \ref{ex:metProjOper}. Again, following Theorems \ref{th:main2} and \ref{th:main3}, we have the following examples of subgradient projection algorithmic operators which determine our outer approximation method:
\begin{enumerate}[a)]
  \item cyclic subgradient projection operator: $T_k=P_{f_{[k]}}$, where $[k]=(k \text{ mod }m)+1$;
  \item most-violated constraint subgradient projection operator: $T_k:=P_{f_{i_k}}$, where $i_k=\argmax_{i\in I_k}f_i^+(x^k)$; see Theorem \ref{th:main3} and Lemma \ref{th:regularityOfSubPro};
  \item simultaneous subgradient projection operator: $T_k:=\frac{1}{|I_k|}\sum_{i\in I_k}P_{f_i}$;
  \item composition subgradient projection operator: $T_k:=\frac 1 2 (\id +\prod_{i\in I_k}P_{f_i})$.
\end{enumerate}
\end{example}

\begin{example}[Augmented block size]\label{ex:augmentedBlock}
Using algorithmic operators over a block of size smaller than $m$ is of practical importance when $m$ is a large number. Therefore we propose to slightly modify the definition of $I_k$ from Examples \ref{ex:metProjOper} and \ref{ex:subProjOper} to obtain an \textit{augmented block}, where $|I_k|=b_k\geq b$. Indeed, we define $I_k$ in a similar ``cyclic'' order, but for the simultaneous and maximum proximity algorithmic operators we want $I_k$ to satisfy
\begin{equation}\label{eq:augmented1}
  |\{i\in I_k \mid \text{constraint $i$ is active at } x^k\}|=b
\end{equation}
if the number of active constraints is greater than or equal $b$. If this is not possible, then we simply set $I_k:=I$. The case of composition methods is slightly different, where we demand that
\begin{equation}\label{eq:augmented2}
  |\{i_t \in I_k=(i_1,\ldots,i_{b_k}) \mid  \text{constraint $i_t$ is active at } U_{i_{t-1}}\ldots U_1x^k\}|=b
\end{equation}
and we set $I_k:=I$ if condition \eqref{eq:augmented2} cannot be satisfied. Therefore, for a fixed $b$, we denote the size of the augmented block by $|I_k|=b+$. Using algorithmic operators over the augmented block, we may significantly accelerate the outer approximation method as we show in the last section of this paper.
\end{example}

\begin{remark}
We would like to mention that there are many more algorithmic operators $T_k$ available in the literature that one could combine with the outer approximation method; see, for example, the definition of dynamic string averaging projection from \cite{CensorZaslavski2013}, modular string averaging from \cite{ReichZalas2015} and double-layer fixed point algorithm from \cite{KolobovReichZalas2016}. Moreover, there are many more adaptive definitions of the convex combinations coefficients, for example,
\begin{equation}\label{eq:omegaikAdaptiveUi}
\omega_i^k:=\frac{p_i(x^k)}{\sum_{i\in I_k}p_i(x^k)}\quad \text{or}\quad
\omega_i^k:=\frac{\|U_i x^k-x^k\|}{\sum_{i\in I_k}\|U_ix^k-x^k\|}.
\end{equation}
Nevertheless, in order to ease the readability of this paper, we focus only on the maximum proximity, simultaneous and composition variants of the outer approximation method.
\end{remark}

\begin{remark}[Polyhedral case and cyclic control]
Consider the outer approximation method (\ref{eq:def:xk}) combined with a cyclic control with relaxation parameters $\alpha_k$ equal to $1$, that is, $T_k=U_{[k]}$ and $[k]= (k \text{ mod }m)+1$. If for each $i\in I$ the operator $U_i$ is a metric projection onto a half-space $C_i=\{z\in\mathcal H\mid \langle a_i,z \rangle\leq \beta_i\}$, then
for each $k=0,1,2,\ldots,$ we have $H_k=C_{[k]}$ whenever $x^k\notin C_{[k]}$ and $H_k=\mathcal H$ otherwise. Therefore in this case we can rewrite method (\ref{eq:def:xk}) in the following form:
\begin{equation}
x^0\in\mathcal H;\qquad x^{k+1}:=P_{C_{[k]}}\left(x^k-\lambda_k Fx^k\right), \text{ for } k=0,1,2,\ldots.
\end{equation}
\end{remark}

\section{Numerical results}

We consider the following best approximation problem:
\begin{equation}
\text{Find}\ x^*\in \Argmin_{z\in C} \frac 1 2 \|z-a\|^2,
\end{equation}
where $a\in\mathbb R^{20}$ is a given vector and $C:=\{z\in \mathbb R^{20}\mid Az\leq b\}$ for some matrix $A\in \mathbb R^{100 \times 20}$. Thus, for each $i\in I:=\{1,\ldots,100\}$, the subset $C_i=\{z\in\mathbb R^{20}\mid \langle a_i,z\rangle\leq b_i \}$ is a half-space. It is not difficult to see that this problem is equivalent to the variational inequality with $F:=\id-a$ and $C=\bigcap_{i\in I} C_i$. We set
\begin{equation}
p_i(x):=(\langle a_i,x\rangle-b_i)_+; \qquad U_ix:=P_{C_i}x=x-\frac{p_i(x)}{\|a_i\|^2}a_i
\end{equation}
and $\lambda_k:=\frac 1 {k+1}$ for $k=0,1,2,\ldots$. We recall that $(x)_+:=\max \{0, x\}$.

Following Theorems \ref{th:main2} and \ref{th:main3}, we consider the outer approximation method \eqref{eq:def:xk}--\eqref{eq:def:Hk} with the following projection operators (PO):
\begin{itemize}
  \item Cyclic PO: $T_k:=U_{[k]}$, where $[k]:=(k \text{ mod } 100)+1$;
  \item Maximum proximity PO: $T_k:=U_{i_k}$, where $i_k:=\argmax_{i\in I_k}p_i(x^k)$;
  \item Simultaneous PO: $T_k:=\frac 1 {|I_k|}\sum_{i\in I_k} P_{C_i}$;
  \item Composition PO: $T_k:=\frac 1 2 (\id +\prod_{i\in I_k}P_{C_i})$
\end{itemize}
For block algorithms we apply two types of control $\{I_k\}_{k=0}^\infty$. The first one with a fixed block size $|I_k|=b$ and the second one, with augmented block size $|I_k|=b+$; see Examples \ref{ex:metProjOper} and \ref{ex:augmentedBlock}, respectively.

For every algorithm we perform 100 simulations, while sharing the same set of randomly generated test problems. We run every algorithm till it reaches 5000 iterations. After running all of the simulations, for every iterate we compute the error $\|x^k-x^*\|$, where $x^*$ is the given solution provided by MATLAB \texttt{fmincon} solver. In order to compare our algorithms, we consider the quantity
\begin{equation}\label{eq:plotError}
\log_{10}\left(\frac{\|x^k-x^*\|}{\|x^0-x^*\|}\right).
\end{equation}
The bold line in Figures \ref{fig:1}-\ref{fig:6} indicates the median computed for (\ref{eq:plotError}). The ribbon plot represents concentrations of order 20, 40, 60 and 80\% around the median. We plot all the information per every 50 iterative steps.

We present now several observations that we have made after running the numerical simulations.
\begin{enumerate}[a)]
  \item The outer approximation method equipped with the composition algorithmic operator outperforms every other method we have considered;
  \item The convergence speed for the maximum proximity, simultaneous and composition methods is monotone with respect to a block size, that is, the larger the block is, the faster the convergence we can expect. Therefore for $b=100$ we expect the best convergence profile for all of the methods;
  \item The augmented block strategy described in Example \ref{ex:augmentedBlock} accelerates the convergence speed. This acceleration is significant in the simultaneous and composition cases;
  \item There is no need to use large blocks with $b=m$. For the maximum proximity it suffices to take $b=20$ and for augmented version even $b=10+$. Similarly, for composition type methods $b=30$ and $b=20+$ are quite close to the case of $b=100$. This can also be seen for the simultaneous projection operator with $b=50+$.
\end{enumerate}

\begin{figure}[H]
\centering
\includegraphics[bb=0 0 948 681, scale=0.35]{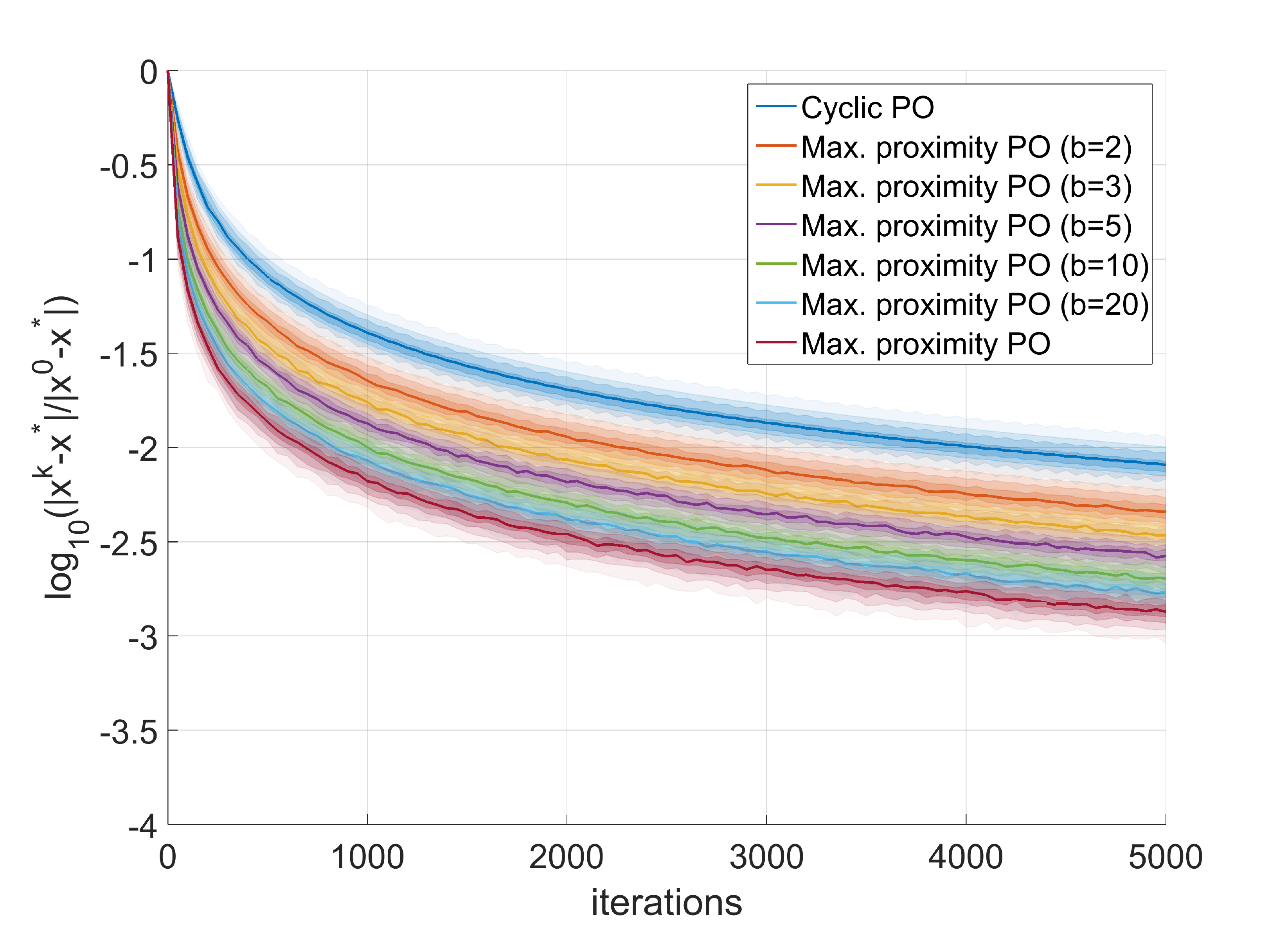}
\caption{Maximum proximity projection operator over the block $I_k$ of size $b=2,3,5,10$ and $20$. For the cyclic algorithm, $b=1$. A bold line indicates the median computed for (\ref{eq:plotError}). The ribbon plot represents concentrations of order 20, 40, 60 and 80\% around the median.}
\label{fig:1}
\end{figure}

\newpage
\begin{figure}[H]
\centering
\includegraphics[bb=0 0 948 681, scale=0.35]{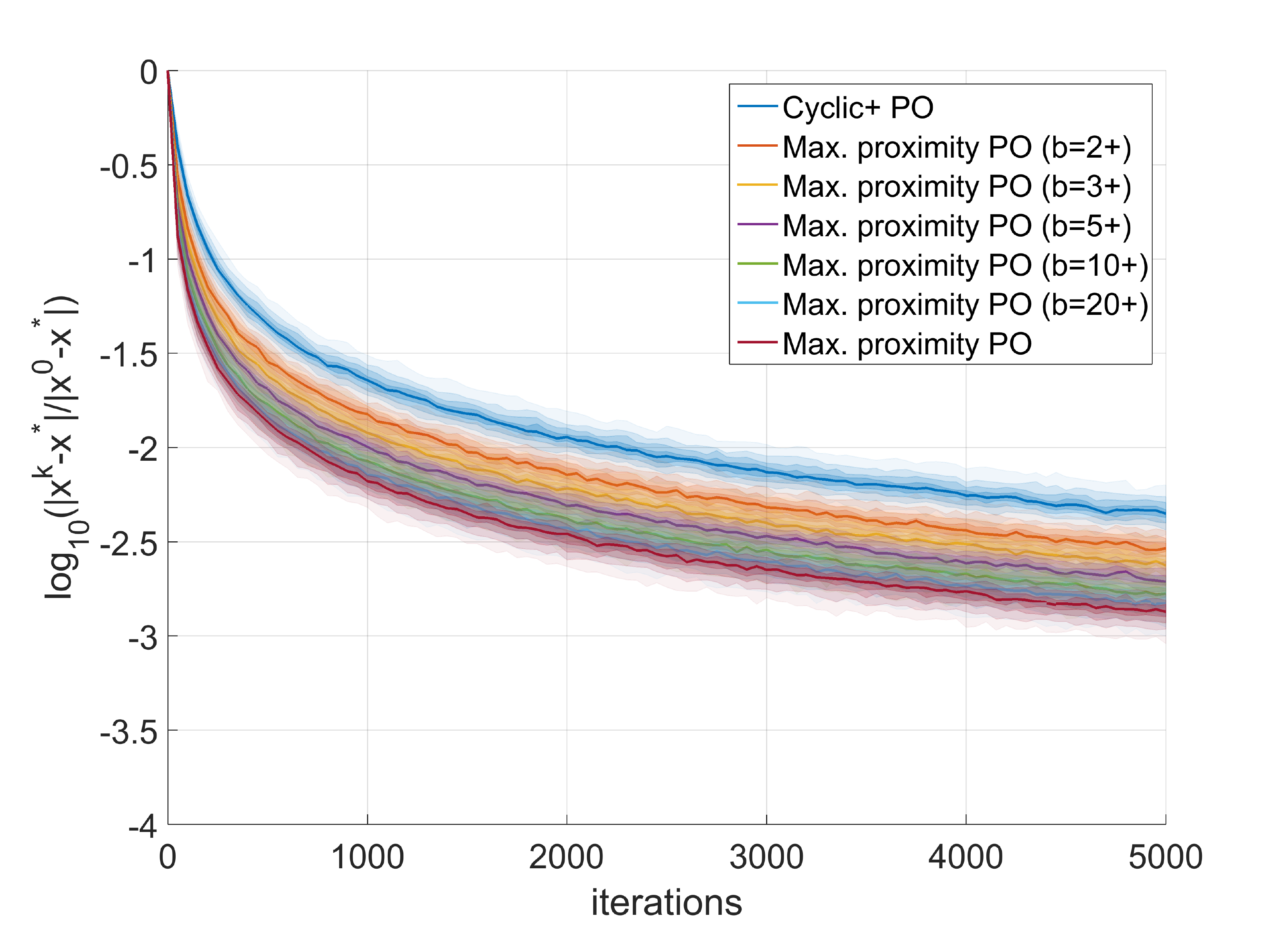}
\caption{Maximum proximity projection operator over the augmented block $I_k$ of size $b=2+,3+,5+,10+$ and $20+$; compare with Example \ref{ex:augmentedBlock}. For the ``cyclic+'' algorithm, $b=1+$. Bold lines and ribbons are the same as in Figure \ref{fig:1}.}
\label{fig:2}
\end{figure}

\begin{figure}[H]
\centering
\includegraphics[bb=0 0 948 681, scale=0.35]{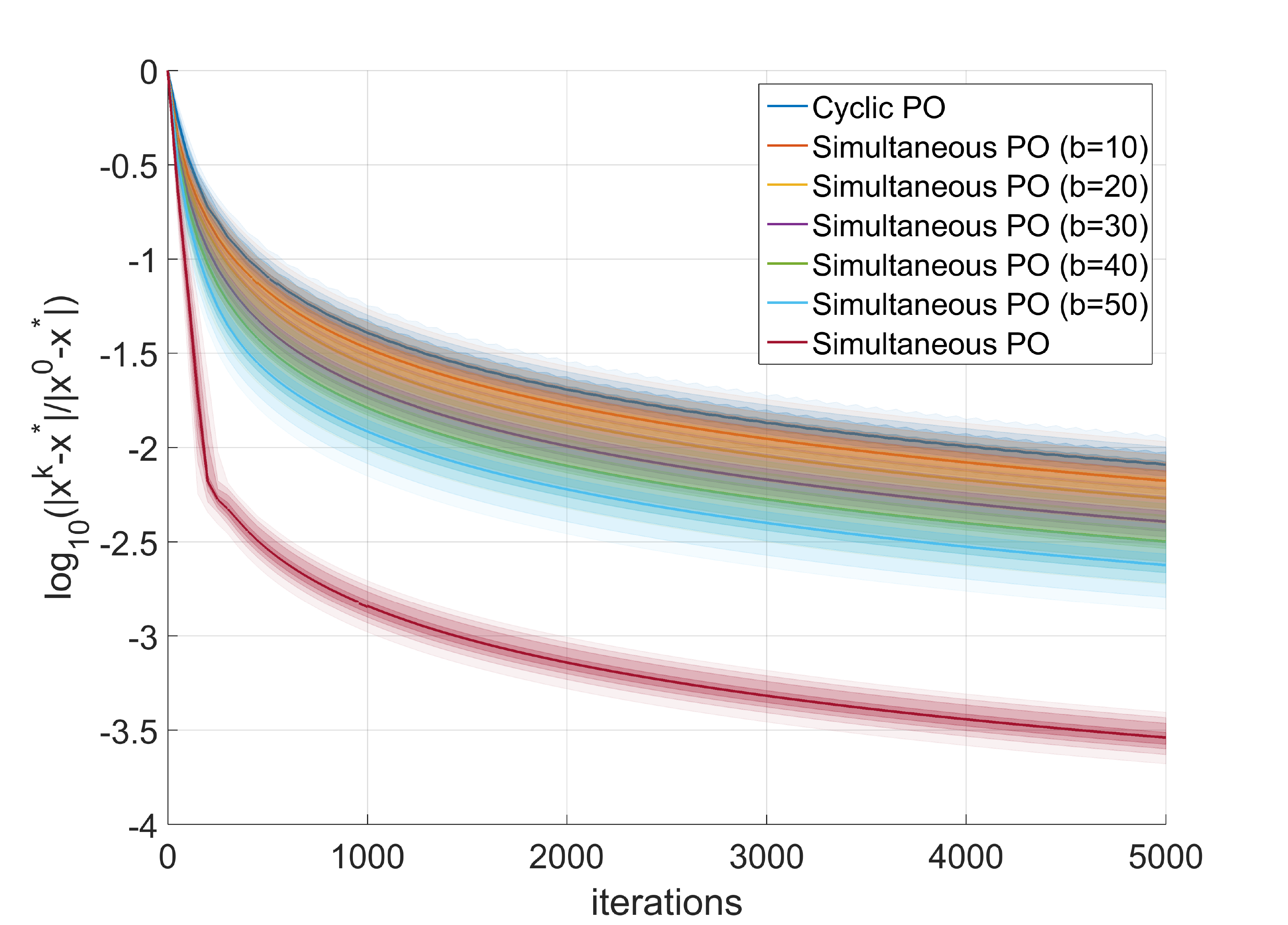}
\caption{Simultaneous projection operator over the block $I_k$ of size $b=10, 20, 30 ,40$ and $50$. For the cyclic algorithm, $b=1$ whereas for the fully simultaneous operator, $b=100$. Bold lines and ribbons are the same as in Figure \ref{fig:1}.}
\label{fig:3}
\end{figure}

\begin{figure}[H]
\centering
\includegraphics[bb=0 0 948 681, scale=0.35]{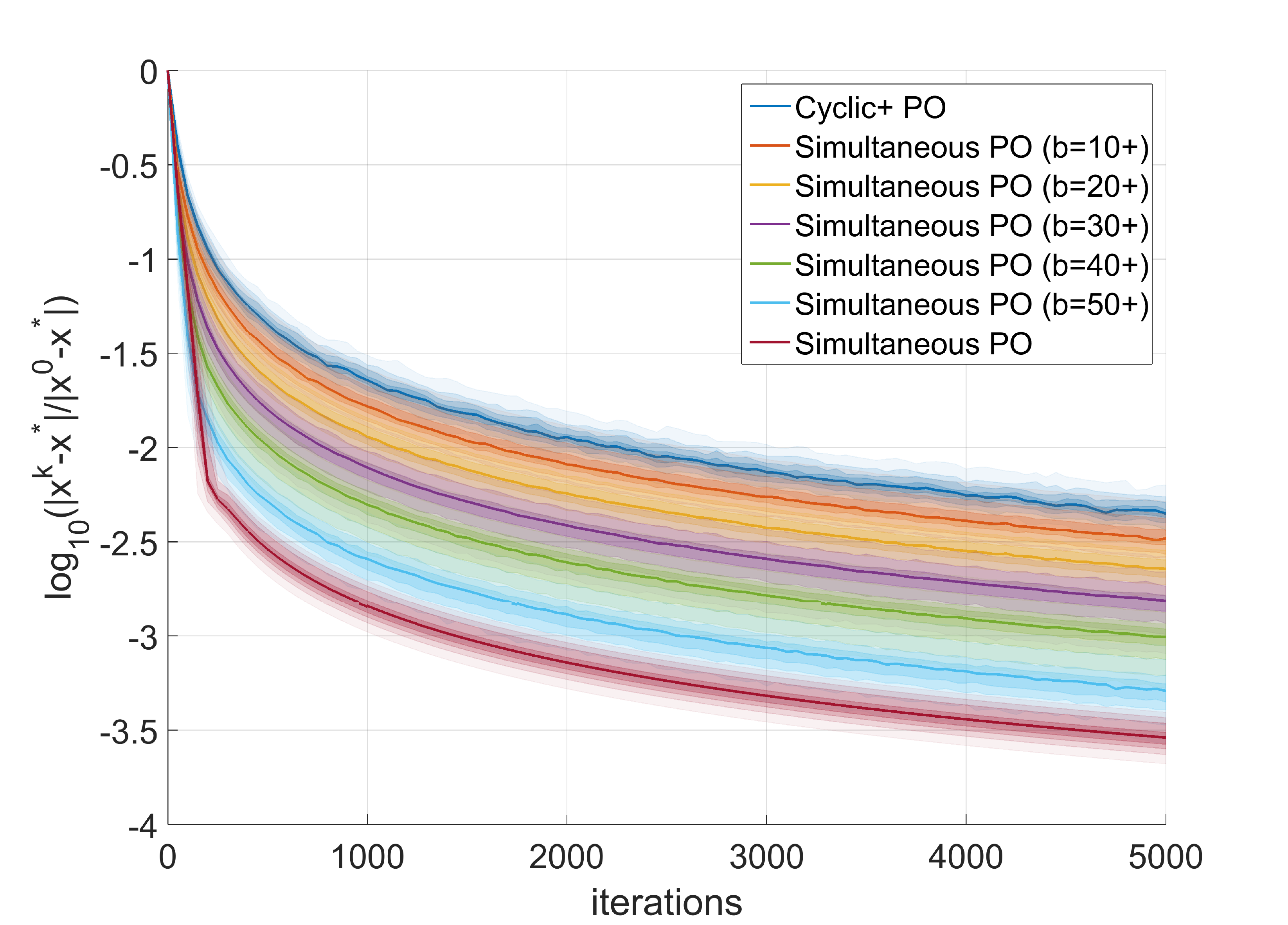}
\caption{Simultaneous projection operator over the augmented block $I_k$ of size $b=1+,10+,20+,30+,40+, 50+$ and $100$; compare with Example \ref{ex:augmentedBlock}. Again, bold lines and ribbons are the same as in Figure \ref{fig:1}.}
\label{fig:4}
\end{figure}

\begin{figure}[H]
\centering
\includegraphics[bb=0 0 948 681, scale=0.35]{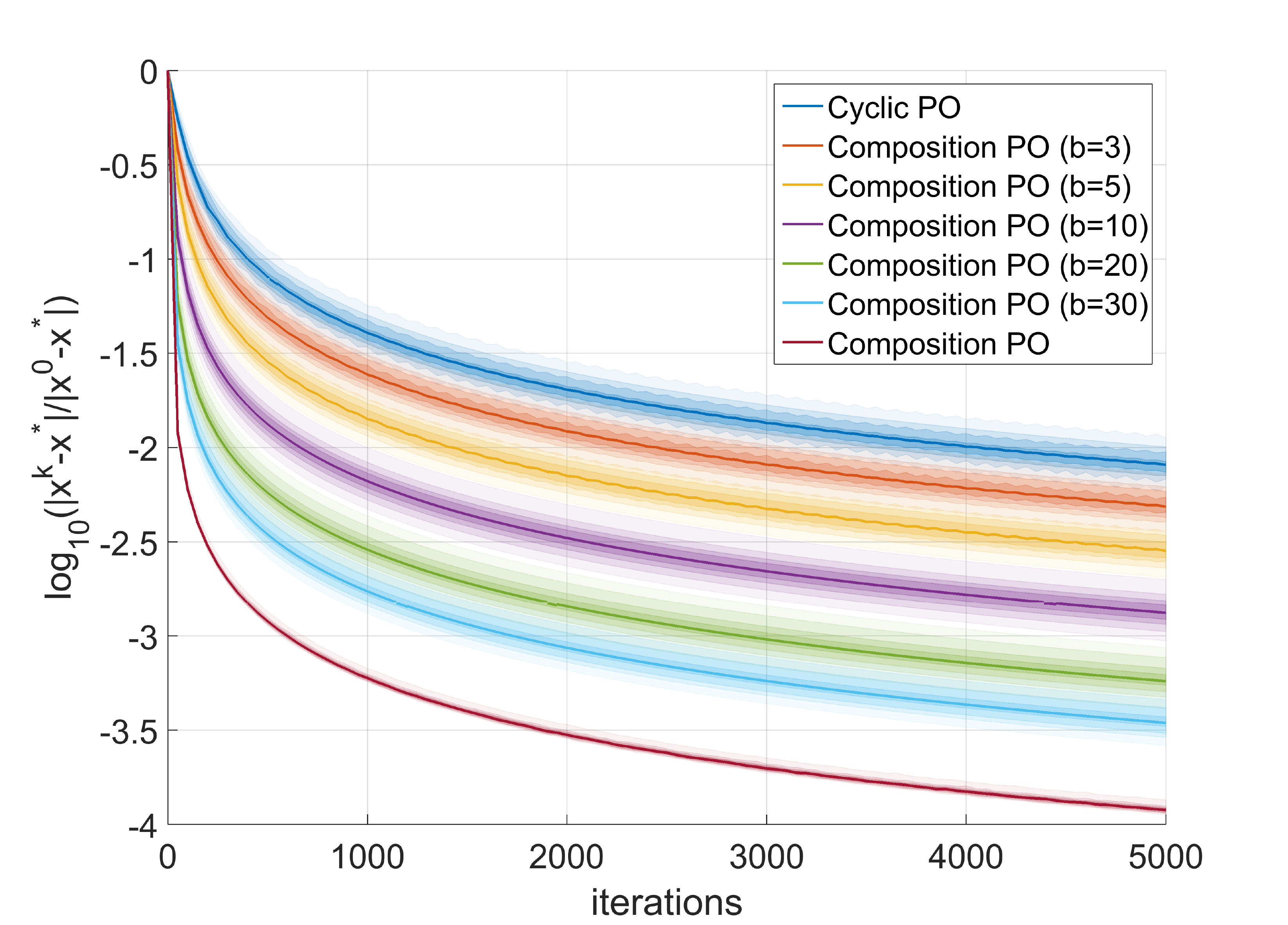}
\caption{Composition projection operator over the block $I_k$ of size $b=3,5,10,20$ and $30$. For the cyclic algorithm, $b=1$ whereas for the full composition operator, $b=100$. Bold lines and ribbons are the same as in Figure \ref{fig:1}.}
\label{fig:5}
\end{figure}

\begin{figure}[H]
\centering
\includegraphics[bb=0 0 948 681, scale=0.35]{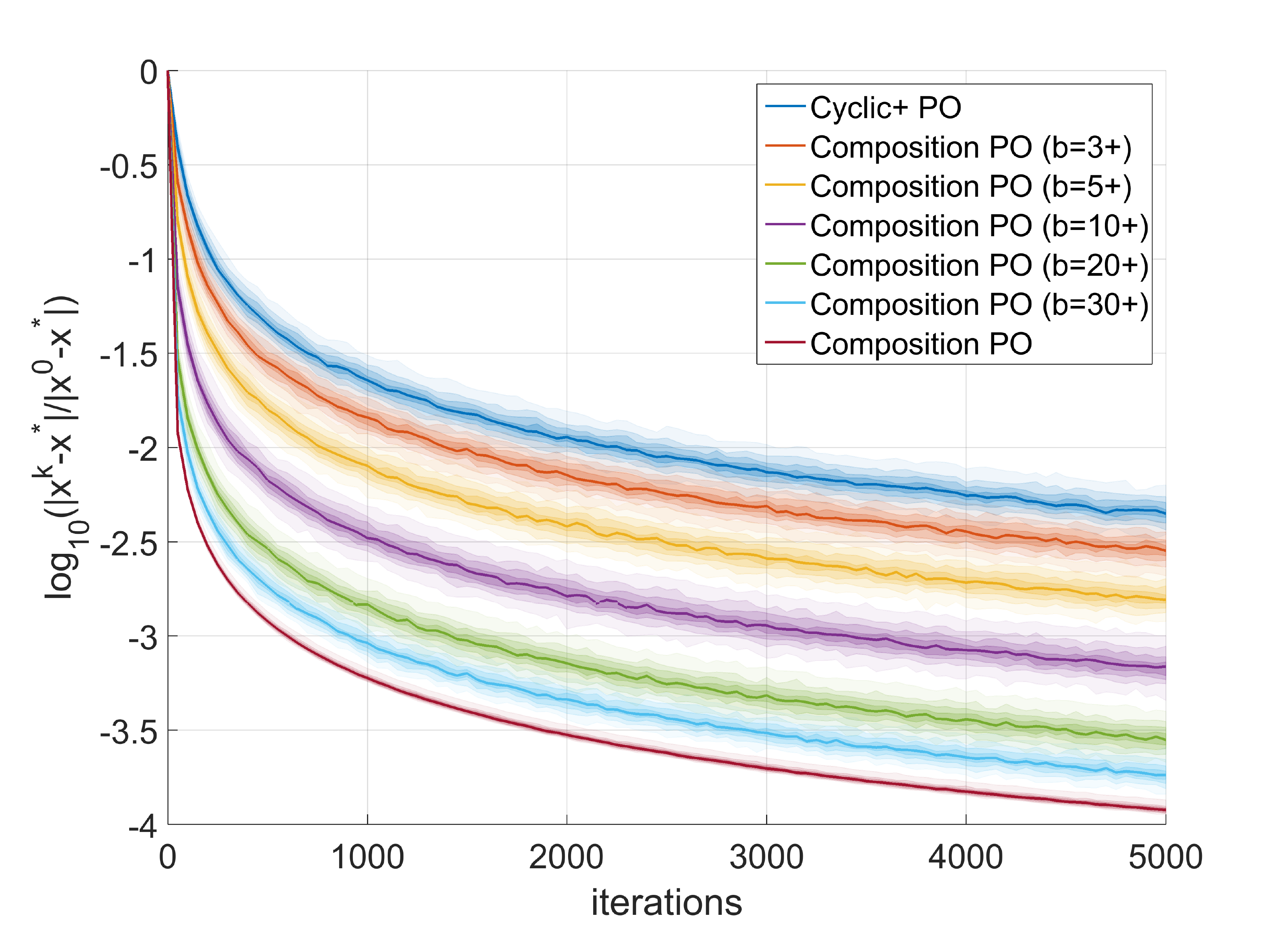}
\caption{Composition projection operator over the augmented block $I_k$ of size $b=1+,10+,20+,30+,40+, 50+$ and $100$; compare with Example \ref{ex:augmentedBlock}. Again, bold lines and ribbons are the same as in Figure \ref{fig:1}.}
\label{fig:6}
\end{figure}

\subsubsection*{Acknowledgments}
We are grateful to the anonymous referee for his/her comments and remarks.

\subsubsection*{Disclosure statement}
No potential conflict of interest was reported by the authors.

\subsubsection*{Funding}
This research was supported in part by the Israel Science Foundation (Grant 389/12), the Fund for the Promotion of Research at the Technion and by the Technion General Research Fund. The third author was financially supported by the Polish National Science Centre within the framework of the Etiuda funding scheme under agreement No. DEC-2013/08/T/ST1/00177.

\bibliographystyle{gOPT}

\begin{thebibliography}{10}

\bibitem{AoyamaKimura2011}
Koji Aoyama and Yasunori Kimura.
\newblock A note on the hybrid steepest descent methods.
\newblock In {\em Fixed point theory and its applications}, pages 73--80. Casa
  C\u ar\c tii de \c Stiin\c t\u a, Cluj-Napoca, 2013.

\bibitem{AoyamaKohsaka2014}
Koji Aoyama and Fumiaki Kohsaka.
\newblock Viscosity approximation process for a sequence of quasinonexpansive
  mappings.
\newblock {\em Fixed Point Theory Appl.}, pages 2014:17, 11, 2014.

\bibitem{Noor2004}
Muhammad Aslam~Noor.
\newblock Some developments in general variational inequalities.
\newblock {\em Appl. Math. Comput.}, 152(1):199--277, 2004.

\bibitem{Bauschke1995}
H.~H. Bauschke.
\newblock A norm convergence result on random products of relaxed projections
  in {H}ilbert space.
\newblock {\em Trans. Amer. Math. Soc.}, 347(4):1365--1373, 1995.

\bibitem{Bauschke1996}
Heinz~H. Bauschke.
\newblock The approximation of fixed points of compositions of nonexpansive
  mappings in {H}ilbert space.
\newblock {\em J. Math. Anal. Appl.}, 202(1):150--159, 1996.

\bibitem{BauschkeBorwein1996}
Heinz~H. Bauschke and Jonathan~M. Borwein.
\newblock On projection algorithms for solving convex feasibility problems.
\newblock {\em SIAM Rev.}, 38(3):367--426, 1996.

\bibitem{BauschkeCombettes2001}
Heinz~H. Bauschke and Patrick~L. Combettes.
\newblock A weak-to-strong convergence principle for {F}ej\'er-monotone methods
  in {H}ilbert spaces.
\newblock {\em Math. Oper. Res.}, 26(2):248--264, 2001.

\bibitem{BauschkeCombettes2011}
Heinz~H. Bauschke and Patrick~L. Combettes.
\newblock {\em Convex analysis and monotone operator theory in {H}ilbert
  spaces}.
\newblock CMS Books in Mathematics/Ouvrages de Math\'ematiques de la SMC.
  Springer, New York, 2011.
\newblock With a foreword by H{\'e}dy Attouch.

\bibitem{BauschkeNollPhan2015}
Heinz~H. Bauschke, Dominikus Noll, and Hung~M. Phan.
\newblock Linear and strong convergence of algorithms involving averaged
  nonexpansive operators.
\newblock {\em J. Math. Anal. Appl.}, 421(1):1--20, 2015.

\bibitem{Bauschke1996PhD}
Heinz~Hermann Bauschke.
\newblock {\em Projection algorithms and monotone operators}.
\newblock PhD thesis, Simon Fraser University, Canada, 1996.

\bibitem{BrowderPetryshyn1966}
F.~E. Browder and W.~V. Petryshyn.
\newblock The solution by iteration of nonlinear functional equations in
  {B}anach spaces.
\newblock {\em Bull. Amer. Math. Soc.}, 72:571--575, 1966.

\bibitem{Cegielski2012}
Andrzej Cegielski.
\newblock {\em Iterative methods for fixed point problems in {H}ilbert spaces},
  volume 2057 of {\em Lecture Notes in Mathematics}.
\newblock Springer, Heidelberg, 2012.

\bibitem{Cegielski2014}
Andrzej Cegielski.
\newblock Extrapolated simultaneous subgradient projection method for
  variational inequality over the intersection of convex subsets.
\newblock {\em J. Nonlinear Convex Anal.}, 15(2):211--218, 2014.

\bibitem{Cegielski2015}
Andrzej Cegielski.
\newblock Application of quasi-nonexpansive operators to an iterative method
  for variational inequality.
\newblock {\em SIAM J. Optim.}, 25(4):2165--2181, 2015.

\bibitem{Cegielski2015b}
Andrzej Cegielski.
\newblock General method for solving the split common fixed point problem.
\newblock {\em J. Optim. Theory Appl.}, 165(2):385--404, 2015.

\bibitem{CegielskiMusallam2016}
Andrzej Cegielski and Fadhel Al-Musallam.
\newblock Strong convergence of a hybrid steepest descent method for the split
  common fixed point problem.
\newblock {\em Optimization}, 65(7):1463--1476, 2016.

\bibitem{CegielskiGibaliReichZalas2013}
Andrzej Cegielski, Aviv Gibali, Simeon Reich, and Rafa{\l} Zalas.
\newblock An algorithm for solving the variational inequality problem over the
  fixed point set of a quasi-nonexpansive operator in {E}uclidean space.
\newblock {\em Numer. Funct. Anal. Optim.}, 34(10):1067--1096, 2013.

\bibitem{CegielskiZalas2013a}
Andrzej Cegielski and Rafa{\l} Zalas.
\newblock Methods for variational inequality problem over the intersection of
  fixed point sets of quasi-nonexpansive operators.
\newblock {\em Numer. Funct. Anal. Optim.}, 34(3):255--283, 2013.

\bibitem{CegielskiZalas2014}
Andrzej Cegielski and Rafa{\l} Zalas.
\newblock Properties of a class of approximately shrinking operators and their
  applications.
\newblock {\em Fixed Point Theory}, 15(2):399--426, 2014.

\bibitem{CensorGibali2008}
Yair Censor and Aviv Gibali.
\newblock Projections onto super-half-spaces for monotone variational
  inequality problems in finite-dimensional space.
\newblock {\em J. Nonlinear Convex Anal.}, 9(3):461--475, 2008.

\bibitem{CensorSegal2009}
Yair Censor and Alexander Segal.
\newblock The split common fixed point problem for directed operators.
\newblock {\em J. Convex Anal.}, 16(2):587--600, 2009.

\bibitem{CensorZaslavski2013}
Yair Censor and Alexander~J. Zaslavski.
\newblock Convergence and perturbation resilience of dynamic string-averaging
  projection methods.
\newblock {\em Comput. Optim. Appl.}, 54(1):65--76, 2013.

\bibitem{ChughRani2014}
Renu Chugh and Rekha Rani.
\newblock Variational inequalities and fixed point problems: a survey.
\newblock {\em International Journal of Applied Mathematical Research},
  3(3):301--326, 2014.

\bibitem{Combettes2001}
Patrick~L. Combettes.
\newblock Quasi-{F}ej\'erian analysis of some optimization algorithms.
\newblock In {\em Inherently parallel algorithms in feasibility and
  optimization and their applications ({H}aifa, 2000)}, volume~8 of {\em Stud.
  Comput. Math.}, pages 115--152. North-Holland, Amsterdam, 2001.

\bibitem{Crombez2006}
G.~Crombez.
\newblock A hierarchical presentation of operators with fixed points on
  {H}ilbert spaces.
\newblock {\em Numer. Funct. Anal. Optim.}, 27(3-4):259--277, 2006.

\bibitem{DeutschYamada1998}
Frank Deutsch and Isao Yamada.
\newblock Minimizing certain convex functions over the intersection of the
  fixed point sets of nonexpansive mappings.
\newblock {\em Numer. Funct. Anal. Optim.}, 19(1-2):33--56, 1998.

\bibitem{FacchineiPang2003}
Francisco Facchinei and Jong-Shi Pang.
\newblock {\em Finite-dimensional variational inequalities and complementarity
  problems. {V}ol. {I} and {II}}.
\newblock Springer Series in Operations Research. Springer-Verlag, New York,
  2003.

\bibitem{Fukushima1986}
Masao Fukushima.
\newblock A relaxed projection method for variational inequalities.
\newblock {\em Math. Programming}, 35(1):58--70, 1986.

\bibitem{GibaliReichZalas2015}
Aviv Gibali, Simeon Reich, and Rafa{\l} Zalas.
\newblock Iterative methods for solving variational inequalities in {E}uclidean
  space.
\newblock {\em J. Fixed Point Theory Appl.}, 17(4):775--811, 2015.

\bibitem{Goldstein1964}
A.~A. Goldstein.
\newblock Convex programming in {H}ilbert space.
\newblock {\em Bull. Amer. Math. Soc.}, 70:709--710, 1964.

\bibitem{Halpern1967}
Benjamin Halpern.
\newblock Fixed points of nonexpanding maps.
\newblock {\em Bull. Amer. Math. Soc.}, 73:957--961, 1967.

\bibitem{Hirstoaga2006}
Sever~A. Hirstoaga.
\newblock Iterative selection methods for common fixed point problems.
\newblock {\em J. Math. Anal. Appl.}, 324(2):1020--1035, 2006.

\bibitem{KolobovReichZalas2016}
Victor~I. Kolobov, Simeon Reich, and Rafał Zalas.
\newblock Weak, strong and linear convergence of a double-layer fixed point
  algorithm.
\newblock Preprint, 2016.

\bibitem{LevitinPolyak1966}
E.S. Levitin and B.T. Polyak.
\newblock Constrained minimization methods.
\newblock {\em USSR Computational Mathematics and Mathematical Physics},
  6(5):1--50, 1966.

\bibitem{Lions1977}
Pierre-Louis Lions.
\newblock Approximation de points fixes de contractions.
\newblock {\em C. R. Acad. Sci. Paris S\'er. A-B}, 284(21):A1357--A1359, 1977.

\bibitem{Opial1967}
Zdzis{\l}aw Opial.
\newblock Weak convergence of the sequence of successive approximations for
  nonexpansive mappings.
\newblock {\em Bull. Amer. Math. Soc.}, 73:591--597, 1967.

\bibitem{PetryshynWilliamson1973}
W.~V. Petryshyn and T.~E. Williamson, Jr.
\newblock Strong and weak convergence of the sequence of successive
  approximations for quasi-nonexpansive mappings.
\newblock {\em J. Math. Anal. Appl.}, 43:459--497, 1973.

\bibitem{ReichZalas2015}
Simeon Reich and Rafa{\l} Zalas.
\newblock A modular string averaging procedure for solving the common fixed
  point problem for quasi-nonexpansive mappings in {H}ilbert space.
\newblock {\em Numer. Algorithms}, 72(2):297--323, 2016.

\bibitem{SlavakisYamadaSakaniwa2003}
Konstantinos Slavakis, Isao Yamada, and Kohichi Sakaniwa.
\newblock Computation of symmetric positive definite toeplitz matrices by the
  hybrid steepest descent method.
\newblock {\em Signal Processing}, 83(5):1135--1140, 5 2003.

\bibitem{TakahashiYamada2008}
Noriyuki Takahashi and Isao Yamada.
\newblock Parallel algorithms for variational inequalities over the {C}artesian
  product of the intersections of the fixed point sets of nonexpansive
  mappings.
\newblock {\em J. Approx. Theory}, 153(2):139--160, 2008.

\bibitem{TakahashiTakeuchiKubota2008}
Wataru Takahashi, Yukio Takeuchi, and Rieko Kubota.
\newblock Strong convergence theorems by hybrid methods for families of
  nonexpansive mappings in {H}ilbert spaces.
\newblock {\em J. Math. Anal. Appl.}, 341(1):276--286, 2008.

\bibitem{Wittmann1992}
Rainer Wittmann.
\newblock Approximation of fixed points of nonexpansive mappings.
\newblock {\em Arch. Math. (Basel)}, 58(5):486--491, 1992.

\bibitem{XiuZhang2003}
Naihua Xiu and Jianzhong Zhang.
\newblock Some recent advances in projection-type methods for variational
  inequalities.
\newblock In {\em Proceedings of the {I}nternational {C}onference on {R}ecent
  {A}dvances in {C}omputational {M}athematics ({ICRACM} 2001) ({M}atsuyama)},
  volume 152, pages 559--585, 2003.

\bibitem{XuKim2003}
H.~K. Xu and T.~H. Kim.
\newblock Convergence of hybrid steepest-descent methods for variational
  inequalities.
\newblock {\em J. Optim. Theory Appl.}, 119(1):185--201, 2003.

\bibitem{Yamada2001}
Isao Yamada.
\newblock The hybrid steepest descent method for the variational inequality
  problem over the intersection of fixed point sets of nonexpansive mappings.
\newblock In {\em Inherently parallel algorithms in feasibility and
  optimization and their applications ({H}aifa, 2000)}, volume~8 of {\em Stud.
  Comput. Math.}, pages 473--504. North-Holland, Amsterdam, 2001.

\bibitem{YamadaOgura2004}
Isao Yamada and Nobuhiko Ogura.
\newblock Hybrid steepest descent method for variational inequality problem
  over the fixed point set of certain quasi-nonexpansive mappings.
\newblock {\em Numer. Funct. Anal. Optim.}, 25(7-8):619--655, 2004.

\bibitem{Zalas2014}
Rafa\l Zalas.
\newblock {\em Variational Inequalities for Fixed Point Problems of
  Quasi-nonexpansive Operators}.
\newblock PhD thesis, University of Zielona G\'ora, Zielona G\'ora, Poland,
  2014.
\newblock In Polish.

\bibitem{Zeidler1985}
Eberhard Zeidler.
\newblock {\em Nonlinear functional analysis and its applications. {III}}.
\newblock Springer-Verlag, New York, 1985.
\newblock Variational methods and optimization, Translated from the German by
  Leo F. Boron.

\bibitem{ZengWongYao2007}
L.~C. Zeng, N.~C. Wong, and J.~C. Yao.
\newblock Convergence analysis of modified hybrid steepest-descent methods with
  variable parameters for variational inequalities.
\newblock {\em J. Optim. Theory Appl.}, 132(1):51--69, 2007.

\bibitem{ZhangHe2013}
Cuijie Zhang and Songnian He.
\newblock A general iterative algorithm for an infinite family of nonexpansive
  operators in {H}ilbert spaces.
\newblock {\em Fixed Point Theory Appl.}, pages 2013:138, 15, 2013.

\end{thebibliography}

\end{document}